\documentclass[11pt]{article}

\usepackage[a4paper,margin=3cm]{geometry}
\usepackage[comma,square,sort&compress,numbers]{natbib}
\usepackage{amsmath, amsthm, amssymb}
\usepackage{graphicx}
\usepackage{caption}
\usepackage{subcaption}
\usepackage{pgf,tikz}
\usepackage{txfonts}
\usetikzlibrary{arrows}
\usepackage[pdfpagemode=UseOutlines]{hyperref}
\hypersetup{
    colorlinks=false,
    pdfborder={0 0 0},
    pdfauthor={B. de Rijk, B. Sandstede},
   pdftitle={Diffusive stability against nonlocalized perturbations of planar wave trains in reaction-diffusion systems}
}
\usepackage{enumitem}

\makeatletter
\def\namedlabel#1#2{\begingroup
    #2%
    \def\@currentlabel{#2}%
    \phantomsection\label{#1}\endgroup
}
\makeatother
\newcommand{\R}{\mathbb{R}}\newcommand{\Z}{\mathbb{Z}}\newcommand{\F}{\mathcal{F}}\newcommand{\C}{\mathbb{C}}

\newcommand{\El}{\mathcal{L}}
\newcommand{\h}{\mathcal{H}}

\newcommand{\Es}{\mathcal{S}}

\newcommand{\A}{\mathcal{A}}

\newcommand{\N}{\mathbb{N}}

\newcommand{\G}{\mathcal{G}}
\newcommand{\GT}{\widetilde{G}}
\newcommand{\vt}{\tilde{v}}

\newcommand{\I}{\mathbb{I}}

\newcommand{\Non}{\mathcal{N}}
\newcommand{\NT}{\widetilde{\mathcal{N}}}
\newcommand{\NTT}{\mathcal{M}}

\numberwithin{equation}{section}

\def\Re{\mathrm{Re}}
\def\Im{\mathrm{Im}}

\let\epsilon\varepsilon

\let\phi\varphi
\def\xx{\zeta}
\def\xt{{\bar{\zeta}}}
\def\yt{{\bar{y}}}
\def\uu{{\hat{u}}}
\def\ub{{\check{u}}}
\def\ut{{\tilde{u}}}
\def\ad{\mathrm{ad}}
\def\per{{\mathrm{per}}}

\newtheoremstyle{theoremsty}{14pt}{14pt}{\itshape}{}{\bfseries}{.}{.5em}{}
\newtheoremstyle{definitionsty}{14pt}{14pt}{\normalfont}{}{\bfseries}{.}{.5em}{}
\theoremstyle{theoremsty}

\newtheorem{theo}{Theorem}[section]
\newtheorem{lem}[theo]{Lemma}

\newtheorem{cor}[theo]{Corollary}

\theoremstyle{definitionsty}

\newtheorem{defi}[theo]{Definition}

\newtheorem{rem}[theo]{Remark}

\makeatletter
\renewenvironment{proof}[1][\proofname] {\par\pushQED{\qed}\normalfont\topsep6\p@\@plus6\p@\relax\trivlist\item[\hskip\labelsep\bfseries#1\@addpunct{.}]\ignorespaces}{\popQED\endtrivlist\@endpefalse} \makeatother

\begin{document}

\title{Diffusive stability against nonlocalized perturbations of planar wave trains in reaction-diffusion systems}

\author{Bj\"orn de Rijk\footnote{\texttt{bjoern.derijk@mathematik.uni-stuttgart.de}, Institut f\"ur Analysis, Dynamik und Modellierung, University of Stuttgart, Pfaffenwaldring 57, 70569 Stuttgart, Germany} \and Bj\"orn Sandstede\footnote{\texttt{bjorn\_sandstede@brown.edu}, Division of Applied Mathematics, Brown University, 182 George Street, Providence, RI 02912, USA. Sandstede was partially supported by the NSF through grant DMS-1714429.}}

\maketitle

\begin{abstract}
Planar wave trains are traveling wave solutions whose wave profiles are periodic in one spatial direction and constant in the transverse direction. In this paper, we investigate the stability of planar wave trains in reaction-diffusion systems. We establish nonlinear diffusive stability against perturbations that are bounded along a line in $\mathbb{R}^2$ and decay exponentially in the distance from this line. Our analysis is the first to treat spatially nonlocalized perturbations that do not originate from a phase modulation. We also consider perturbations that are fully localized and establish nonlinear stability with better decay rates, suggesting a trade-off between spatial localization of perturbations and temporal decay rate. Our stability analysis utilizes pointwise estimates to exploit the spatial structure of the perturbations. The nonlocalization of perturbations prevents the use of damping estimates in the nonlinear iteration scheme; instead, we track the perturbed solution in two different coordinate systems.

\end{abstract}

\section{Introduction}

In this paper, we investigate the stability of spatially periodic planar travelling waves. Consider a planar reaction-diffusion system of the form
\begin{equation}\label{2D}
u_t = D(u_{xx}+u_{yy}) + f(u), \quad (x,y)\in\mathbb{R}^2, \quad t \geq 0, \quad u\in\mathbb{R}^n,
\end{equation}
where $n \in \N$, $D \in \R^{n \times n}$ is a symmetric, positive-definite matrix, and $f \colon \R^n \to \R^n$ is a $C^3$-smooth nonlinearity. We are interested in planar travelling-wave solutions to~\eqref{2D} of the form $u(x,y,t)=u_\infty(kx-\omega t)$, where the profile $u_\infty(\zeta)$ is $C^3$-smooth and periodic in $\zeta$ with period $1$, $k \in \R$ denotes the spatial wave number, and $\omega \in \R$ is the temporal frequency of the travelling wave. From now on, we use the term \emph{wave train} to refer to spatially-periodic travelling waves. We note that the terms ``rolls'' and ``stripes'' are also used in literature to refer to planar wave trains.

Our goal is to determine whether, and in what sense, the planar wave train $u(x,y,t)=u_\infty(kx-\omega t)$ is stable under perturbations of the initial condition $u(x,y,0)=u_\infty(kx)$. Part of our motivation stems from the case of planar spiral waves that resemble planar wave trains in the far field: understanding the stability of wave-train solutions to~\eqref{2D} is a first step towards any nonlinear stability analysis of planar spiral waves.

Before discussing the nonlinear stability of wave trains for the planar system~\eqref{2D}, we review the relevant results for the spatially one-dimensional case. Note that the function $u(x,t)=u_\infty(kx-\omega t)$ is also a wave-train solution to the one-dimensional version
\begin{equation}\label{1D}
u_t = D u_{xx} + f(u), \quad x\in\mathbb{R}, \quad t \geq 0, \quad u \in \R^n,
\end{equation}
of~\eqref{2D}. Throughout, we will assume that the wave train is spectrally stable and refer to~\S\ref{s2.1} for details on what this assumption entails. We then consider initial conditions of the form
\begin{equation}\label{e3}
\tilde{u}(x,0) = u_\infty(kx+\phi_0(x)) + v_0(x), \qquad \phi_0(x)\to\phi_\pm\mbox{ as }x\to\pm\infty,
\end{equation}
where the perturbation $v_0$ is sufficiently small in an appropriate function space, so that we change the phase, but not the wave number, of the wave train at time $t=0$. Let $\tilde{u}(x,t)$ denote the associated solution to~\eqref{1D}: we may then ask whether $\tilde{u}(x,t)$ converges in an appropriate sense to $u_\infty(kx-\omega t)$, or a translate, as time $t$ goes to infinity.

More generally, we can attempt to write the solution in the form
\begin{equation}\label{e4}
\tilde{u}(x,t) = u_\infty(kx+\phi(x,t)-\omega t) + \mbox{ terms that decay at least pointwise in time}.
\end{equation}
For the case $|\phi_+-\phi_-| \ll 1$, it was shown in~\cite{SAN3} that~\eqref{e4} holds for a function $\phi(x,t)$ that has an asymptotically self-similar profile as $t \to \infty$: indeed, $\phi(x,t)$ converges to a moving Gaussian if $\phi_+=\phi_-$ and to a moving error function with amplitude $\phi_+-\phi_-$ in the case where $0<|\phi_+-\phi_-|\ll1$. Similar results, though without the explicit asymptotics, were also shown in~\cite{JONZNL,JONZ,JUN,JUNNL,WUS} using different methods -- see Remark~\ref{litoverview} below for more details. The results in~\cite{JONZNL} were complemented with explicit asymptotics in~\cite{JONZNL2}, recovering the results of~\cite{SAN3} for the case $0 < |\phi_+ - \phi_-| \ll 1$. The restriction that $|\phi_+-\phi_-|$ is small was recently removed in~\cite{IS}. We emphasize that, although the initial phase off-set $\phi_0$ can be nonlocalized, the perturbation $v_0$ in~\eqref{e3} has to be localized in all the aforementioned papers, that is, we need to assume that $v_0(x) \to 0$ sufficiently rapidly as $x \to \pm \infty$.

In this paper, we examine the nonlinear stability of planar wave trains $u(x,y,t)=u_\infty(kx-\omega t)$ that satisfy~\eqref{2D}. We note that increasing the spatial dimension from one to two improves the decay properties on the linear level: for instance, the solutions to the heat equation
\begin{equation*}
u_t = \Delta_{{\bf{x}}} u, \quad {\bf{x}} \in \R^d, \quad t \geq 0, \quad u \in \R^n,
\end{equation*}
decay pointwise with rate $t^{-d/2}$ so that increasing the spatial dimension $d\in\N$ yields faster decay in time. This leads to the natural question whether the additional decay can be exploited to allow for a larger -- or different -- class of initial conditions. In this paper, we answer this question in the affirmative by proving that planar wave trains are stable against a class of nonlocalized perturbations $v_0(x,y)$ rendering our initial conditions complementary to those considered in the literature for one spatial dimension -- see Remark~\ref{complement} below for more details.

More specifically, given initial conditions of the form
\begin{equation}\label{e23}
\tilde{u}(x,y,0) = u_\infty(kx) + v_0(x,y),
\end{equation}
we prove that the resulting solution $\tilde{u}(x,y,t)$ to~\eqref{2D} decays pointwise with rate $t^{-1/2}$ like a diffusive Gaussian if $v_0(x,y)$ is bounded along an arbitrary, but fixed, line in $\mathbb{R}^2$ and decays exponentially in the distance from this line. We note that the obtained decay is optimal for this class of perturbations -- see Remark~\ref{optimality} below -- although it is slower than one would intuitively expect suggesting a trade-off between spatial localization of the initial perturbation and temporal decay rate. Indeed, our second result recovers the expected pointwise decay with rate $t^{-1}$ provided $v_0(x,y)$ is sufficiently localized in all spatial directions in $\R^2$. Although this second theorem provides an expected result, it seems not to be present in the literature at the moment -- see Remark~\ref{litoverview2} below.

The proofs of our main results depend on pointwise estimates that allow us to exploit the specific spatial structure of the perturbations. Therefore, we extend the pointwise Green's function estimates  for wave trains proved in~\cite{JUN} in one spatial dimension to the planar case. However, the nonlinear stability analysis in~\cite{JUN} -- which is, in turn, based on~\cite{JONZ} -- does not extend to our planar case: the $L^2$-localization of the initial perturbation is utilized in~\cite{JONZ,JUN} in a crucial way in their nonlinear damping estimates, which are necessary to close the nonlinear iteration arguments presented in~\cite{JONZ,JUN}. The main challenge in our proof is to replace these damping estimates: we rely instead on tracking the perturbed solution in two different coordinate systems, where one of these coordinates is used to correct for the phase shift of the solution relative to the original wave train.

We now comment briefly on open problems. A major open problem in the one-dimensional case is the long-time dynamics of initial conditions of the form $\tilde{u}(x,0) = u_\infty(k_0(x)x)$, which correspond to perturbations of the wave number instead of just the phase. In the planar case, not much is known about the linear and nonlinear stability of spiral waves. Moreover, it is unexplored whether, and in what form, planar wave trains are stable against an initial change in phase, that is whether, and in what form, stability holds for initial conditions of the form $\tilde{u}(x,y,0) = u_\infty(kx + \phi_0(x,y)) + v_0(x,y)$ with $v_0$ possibly nonlocalized. Finally, one expects even higher decay rates in the stability analysis of wave trains in spatial dimensions $d \geq 3$, and one could therefore attempt studying even larger classes of initial conditions in this setting.

\begin{rem} \label{litoverview}
There are various methods to prove nonlinear stability results of wave-train solutions to reaction-diffusion systems in one spatial dimension. The proofs in~\cite{IS,SAN3} are based on a decomposition of phase and amplitude variables as in~\cite{DSSS}, mode filters to separate critical (translational) modes from noncritical exponentially decaying modes, and on the renormalization group techniques developed in~\cite{BKL,SCH2}. On the other hand, in~\cite{JONZNL,JONZ}, $L^p$-estimates on the Green's function (again obtained by separating critical from noncritical modes) are applied in a nonlinear iteration scheme that was originally devised in shock-wave theory. The set-up in~\cite{JUN,JUNNL} follows~\cite{JONZNL,JONZ}, but employs pointwise Green's function estimates instead. The use of pointwise Green's function estimates was introduced in~\cite{ZUH} and further extended in~\cite{OHZUM} to the setting of wave trains: it has the advantage of resulting in detailed information on both temporal and spatial decay. Finally, in~\cite{WUS}, a normal form about the wave train was constructed that arises as a conjugation of the reaction-diffusion system~\eqref{1D} with a lattice dynamical system of discrete phase equations. The normal form exhibits a conservation law associated with the translational symmetry, which can be exploited to gain additional decay in the nonlinear stability argument.
\end{rem}

\begin{rem}\label{litoverview2}
Although this paper seems to be the first to consider the nonlinear stability of wave trains in planar reaction-diffusion systems~\eqref{2D}, the methods for proving nonlinear stability in one spatial dimension -- see Remark~\ref{litoverview} -- have been applied in higher space dimensions to different systems. For instance, the renormalization approach from~\cite{SCH2} is used in~\cite{UECR} to prove that the planar Swift-Hohenberg equation admits nonlinearly stable wave-train solutions. In addition, the methods in~\cite{JONZNL,JONZ} are applied in higher space dimensions in~\cite{JONZVD12,OHZVD3} to prove nonlinear stability results for wave trains in viscous systems of conservation laws. We emphasize that, unlike in the current paper, localized perturbations were considered in these references.
\end{rem}

\begin{rem} \label{complement}
In this paper, we consider initial conditions of the form~\eqref{e23}, where the perturbation $v_0(x,y)$ is bounded along a line in $\mathbb{R}^2$ and decays exponentially in the distance from this line. All initial conditions for one spatial dimension considered in the literature -- see~\cite{JONZNL,JONZ,JUN,JUNNL,WUS,SAN3} -- are of the form~\eqref{e3}, where the perturbation $v_0(x)$ is sufficiently localized and the initial phase off-set $\phi_0(x)$ satisfies $\phi_0(x) \to \phi_\pm$ as $x \to \pm \infty$. This implies that, in the limit $x \to \pm \infty$, the initial condition $\tilde{u}(x,0)$ converges to the translate $u_\infty(kx + \phi_\pm)$ of the wave train. In our case, the initial condition~\eqref{e23} does not (necessarily) converge to a translate of the wave train when letting $|(x,y)| \to \infty$ over some line in $\R^2$. Indeed, we allow for instance for the perturbation $v_0(x,y) = \epsilon e^{-y^2}$ with $0 < \epsilon \ll 1$. Thus, the nonlocalization of the perturbations in this paper does not originate from a phase modulation, and our class of initial conditions is therefore complementary to those present in literature.
\end{rem}

\section{Main results} 

In this paper, we establish nonlinear diffusive stability of the planar wave-train solution $u(x,y,t) = u_\infty(kx-\omega t)$ to~\eqref{2D} against a class of spatially localized perturbations and a class of spatially nonlocalized perturbations under the assumption that the wave train is spectrally stable. Our nonlinear stability results and the associated spectral assumptions are most naturally formulated by switching to a co-moving frame $\xx = kx - \omega t$, which yields a \emph{stationary} wave-train solution $u(\xx,y,t) = u_\infty(\xx)$ to
\begin{align} u_t = D\left(k^2 \partial_{\xx\xx} + \partial_{yy}\right)u + \omega \partial_\xx u + f(u), \quad (\xx,y) \in \R^2, \quad t \geq 0, \quad u \in \R^n, \label{comove}\end{align}
Before stating our nonlinear stability results, we disclose what spectral stability entails.

\subsection{Spectral stability}\label{s2.1}

The linearization of~\eqref{comove} about $u(\xx,y,t) = u_\infty(\xx)$ is
\begin{align} u_t = \El u, \qquad \El u = D\left(k^2 \partial_{\xx\xx} + \partial_{yy}\right) u + \omega \partial_\xx u + f'(u_\infty(\xx)) u, \label{lin}\end{align}
where $\El$ is a sectorial operator on $L^2(\R^2,\C^n)$ -- see~\cite{HEN,LUN} and note that $\El$ is elliptic as $D$ is a symmetric positive-definite matrix. We apply the Fourier transform $\hat{ \ }$, given by
\begin{align*} \uu(\xx,\nu_y) = \frac{1}{\sqrt{2\pi}}\int_{\R} e^{-i\nu_y y} u(\xx,y)dy,\end{align*}
in the transverse coordinate $y$ to the linearization~\eqref{lin} and obtain
\begin{align} \uu_t = \El_{\nu_y} \uu, \qquad \El_{\nu_y} \uu  = D\left(k^2 \partial_{\xx\xx} - \nu_y^2\right) \uu + \omega \partial_\xx \uu + f'(u_\infty(\xx)) \uu, \label{Fourier}\end{align}
where $\El_{\nu_y}$ is a sectorial operator on $L^2(\R,\C^n)$ for each $\nu_y \in \R$. The solution operator is represented via the inverse Fourier transform as
\begin{align} [e^{\El t} u](\xx,y) = \frac{1}{\sqrt{2\pi}}\int_{\R} e^{i\nu_y y} e^{\El_{\nu_y} t} \uu(\xx,\nu_y)d\nu_y. \label{interm}\end{align}
Since $\El_{\nu_y}$ is a periodic differential operator on $L^2(\R,\C^n)$, we can apply the Bloch transform to~\eqref{Fourier} -- see~\cite{DSSS,REE}. The composition $\check{ \ }$ of the Bloch and Fourier transforms is given by
\begin{align*} \ub(\xx,{\pmb \nu}) = \sum_{j \in \Z} e^{2\pi i j \xx} \F_2(u)(\nu_x + 2\pi j,\nu_y), \qquad \F_2(u)({\pmb \nu}) = \frac{1}{2\pi} \int_{\R^2} e^{-i{\pmb \nu} \cdot {\pmb \varsigma}} u(\xx,y)d\xx dy,\end{align*}
where $\F_2$ is the Fourier transform on $\R^2$, $\cdot$ denotes the dot product, ${\pmb \nu} = (\nu_x,\nu_y)$ and ${\pmb \varsigma} = (\xx,y)$. Applying the Fourier-Bloch transform to~\eqref{lin} yields
\begin{align} \ub_t = L_{{\pmb \nu}} \ub, \qquad L_{{\pmb \nu}} \ub = D\left(k^2 \left(\partial_{\xx} + i\nu_x\right)^2 - \nu_y^2\right) \ub + \omega (\partial_\xx + i\nu_x) \ub + f'(u_\infty(\xx)) \ub, \label{Bloch2}\end{align}
where $L_{\pmb \nu}$ is a sectorial operator on $L_\per^2([0,1],\C^n)$ with compact resolvent for each ${\pmb \nu} \in \Omega := [-\pi,\pi] \times \R$. The solution operator is represented via the inverse Fourier-Bloch transform as
\begin{align} [e^{\El t} u](\xx,y) = \frac{1}{2\pi} \int_{\Omega} e^{i{\pmb \nu} \cdot {\pmb \varsigma}} e^{L_{{\pmb \nu}} t} \ub({\pmb \nu},\xx)d{\pmb \nu}. \label{Bloch} \end{align}
Furthermore, we obtain the spectral decomposition
\begin{align} \sigma(\El) = \bigcup_{{\pmb \nu} \in \Omega} \sigma(L_{\pmb \nu}), \label{specdecomp}\end{align}
where the spectra $\sigma(L_{\pmb \nu})$ are discrete for each ${\pmb \nu} \in \Omega$ as $L_{\pmb \nu}$ has compact resolvent.

Since $u_\infty'$ is a solution to~\eqref{lin}, $0$ is an eigenvalue of $L_0$. Now, provided $0$ is simple as an eigenvalue of $L_0$, the implicit function theorem provides a surface $\lambda_0 \colon U \to \C$, where $U \subset \C^2$ is a neighborhood of $0$, such that $\lambda_0(0) = 0$ and $\lambda_0({\pmb \nu})$ is a simple eigenvalue of $L_{\pmb \nu}$ for any ${\pmb \nu} \in U$. By symmetry of the spectrum of the real operator $\El$ and the decomposition~\eqref{specdecomp}, the surface $\lambda_0[\Omega] \subset \sigma(\El)$ touches the imaginary axis generically quadratically. So, even in the most stable scenario, where the spectrum of $\El$ is bounded away to the left of the imaginary axis except for a quadratic touching at the origin, there is no spectral gap. This leads to the following definition.

\begin{defi}
The planar wave train $u(x,y,t)=u_\infty(kx-\omega t)$ to~\eqref{2D} is \emph{spectrally stable} if there exists $\eta,\epsilon > 0$ such that
\begin{itemize}
\item[(D1)] $L_0$ has no spectrum in $\Re(\lambda) \geq -\eta$ besides a simple eigenvalue $\lambda = 0$;
\item[(D2)] It holds $\Re(\sigma(L_{\pmb \nu})) \leq -\eta |{\pmb \nu}|^2$ for all ${\pmb \nu} \in \Omega$ with $|{\pmb \nu}| \leq \epsilon$;
\item[(D3)] It holds $\Re(\sigma(L_{\pmb \nu})) < -\eta$ for all ${\pmb \nu} \in \Omega$ with $|{\pmb \nu}| \geq \epsilon$.
\end{itemize}
\end{defi}

Throughout this paper we require the spectral stability assumptions (D1)-(D3) to hold true. Similar spectral assumptions are made in one spatial dimension -- see~\cite{JONZ,JUN,SAN3}. We emphasize that spectrally-stable planar wave trains to~\eqref{2D} can be generated from spectrally-stable wave trains to~\eqref{1D} -- see~\S\ref{specasscon}. Examples of such wave trains in one spatial dimension can be found in~\cite{BDR2,SAS}. Note that by Sturm-Liouville theory wave-train solutions to~\eqref{1D} can only be spectrally stable if $n > 1$, i.e.~if~\eqref{1D} is a proper \emph{system} of reaction-diffusion equations.

\subsection{Statement of results} \label{sec:statres}

Our results concern nonlinear diffusive stability of spectrally-stable planar wave-train solutions $u(x,y,t)=u_\infty(kx-\omega t)$ to~\eqref{2D} against classes of spatially localized and spatially nonlocalized perturbations. In our first result, we take a unit vector ${\pmb w} \in \R^2$ and consider nonlocalized $W^{2,\infty}$-perturbations that are bounded, when restricted to the family of lines
\begin{align} \left\{{\pmb z} \in \R^2 \colon {\pmb w} \cdot {\pmb z} = a\right\}, \qquad a \in \R, \label{lines}\end{align}
but that are exponentially localized in any other spatial direction. We establish diffusive Gaussian-like decay of the perturbation and its derivatives with rate $t^{-1/2}$. In addition, the perturbation stays bounded for all times $t \geq 0$ on lines of the form~\eqref{lines} and exponentially localized in any other spatial direction. If we account for translational invariance and allow for a phase shift, it is possible to obtain decay with rate $\log(t)/t$. Thus, our result is as follows.

\begin{theo} \label{maintheorem}
Assume (D1)-(D3) hold true and let $(\beta,\gamma) \in \R^2$ be a unit vector. There exists constants $E_0>0$ and $C,M > 1$ such that for all $v_0 \in W^{2,\infty}(\R^2,\R^n)$ satisfying
\begin{align} \sup_{\xx, y \in \R} e^{\frac{|\beta \xx + \gamma y|^2}{M}} \left(\left\|v_0(\xx,y)\right\| + \left\|\partial_\xx v_0(\xx,y)\right\| + \left\|\partial_{\xx\xx} v_0(\xx,y)\right\|\right) \leq E_0, \label{v01}\end{align}
there exists a solution $\ut(t)$ to~\eqref{comove}, with initial condition $\ut(0) = u_\infty + v_0$, for all time $t \geq 0$ satisfying
\begin{align*}
\left\|\ut(\xx,y,t) - u_\infty(\xx)\right\| \leq C\frac{e^{-\frac{|\beta \xx + \gamma y + \alpha \beta t|^2}{M(1 + t)}}}{\sqrt{1+t}}, \qquad \left\|D_{\xx,y}^{\mathfrak{c}}\left(\ut(\xx,y,t) - u_\infty(\xx)\right)\right\| \leq C\frac{e^{-\frac{|\beta \xx + \gamma y + \alpha \beta t|^2}{M(1 + t)}}}{\sqrt{t}},
\end{align*}
for $\xx,y \in \R$ and $t > 0$, where $\alpha$ is as in Lemma~\ref{Lem1} and $\mathfrak{c} \in \Z_{\geq 0}^2$ is a multi-index with $|\mathfrak{c}| = 1$. In addition, there exists a function $\psi \colon \R_{\geq 0} \to  W^{3,\infty}(\R^2,\R)$ such that
\begin{align*}
\left\|\ut(\xx + \psi(\xx,y,t),y,t) - u_\infty(\xx)\right\| &\leq C\frac{\log(2+t)}{1+t} e^{-\frac{|\beta \xx + \gamma y + \alpha \beta t|^2}{M(1 + t)}},\\
\left\|\psi(\xx,y,t)\right\| \leq \frac{e^{-\frac{|\beta \xx + \gamma y + \alpha \beta t|^2}{M(1 + t)}}}{\sqrt{1+t}}, & \qquad \left\|D_{\xx,y,t}^{\mathfrak{b}} \psi(\xx,y,t)\right\| \leq \frac{e^{-\frac{|\beta \xx + \gamma y + \alpha \beta t|^2}{M(1 + t)}}}{1+t},
\end{align*}
for $\xx,y \in \R, t \geq 0$ and $\mathfrak{b} \in \Z_{\geq 0}^3$ with $1 \leq |\mathfrak{b}| \leq 3$.
\end{theo}

Next, we study the stability of the planar wave-train solution $u(x,y,t)=u_\infty(kx-\omega t)$ to~\eqref{2D} against exponentially localized perturbations. We obtain diffusive Gaussian-like decay of the perturbation with rate $t^{-1}$, which can be improved to $t^{-3/2}$ by tracking the phase shift. Moreover, the perturbation stays exponentially localized for all time $t \geq 0$. This leads to the following result.

\begin{theo} \label{maintheorem2}
Assume (D1)-(D3) hold true. There exists constants $E_0>0$ and $C,M > 1$ such that for all $v_0 \in W^{1,\infty}(\R^2,\R^n)$ satisfying
\begin{align} \sup_{\xx, y \in \R} e^{\frac{\xx^2 + y^2}{M}} \left(\left\|v_0(\xx,y)\right\| + \left\|\partial_\xx v_0(\xx,y)\right\|\right) \leq E_0, \label{v02}\end{align}
there exists a solution $\ut(t)$ to~\eqref{comove}, with initial condition $\ut(0) = u_\infty + v_0$, for all time $t \geq 0$ satisfying
\begin{align*}
\left\|\ut(\xx,y,t) - u_\infty(\xx)\right\| &\leq C\frac{e^{-\frac{|\xx + \alpha t|^2 + y^2}{M(1 + t)}}}{1+t}, \qquad \left\|\partial_\xx \left(\ut(\xx,y,t) - u_\infty(\xx)\right)\right\| \leq C e^{-\frac{|\xx + \alpha t|^2 + y^2}{M(1 + t)}}\frac{\log(2+t)}{\sqrt{t}\sqrt{1+t}},\\
&\left\|\partial_y \ut(\xx,y,t)\right\| \leq C e^{-\frac{|\xx + \alpha t|^2 + y^2}{M(1 + t)}}\frac{\log(2+t)}{\sqrt{t}(1+t)},
\end{align*}
for $\xx,y \in \R$ and $t > 0$, where $\alpha$ is as in Lemma~\ref{Lem1} and $\mathfrak{c} \in \Z_{\geq 0}^2$ is a multi-index with $|\mathfrak{c}| = 1$. In addition, there exists a function $\psi \colon \R_{\geq 0} \to  W^{3,\infty}(\R^2,\R)$ such that
\begin{align*}
\left\|\ut(\xx + \psi(\xx,y,t),y,t) - u_\infty(\xx)\right\| &\leq C\frac{e^{-\frac{|\xx + \alpha t|^2 + y^2}{M(1 + t)}}}{(1+t)\sqrt{1+t}},\\
\left\|\psi(\xx,y,t)\right\| \leq C\frac{e^{-\frac{|\xx + \alpha t|^2 + y^2}{M(1 + t)}}}{1+t}, & \qquad \left\|D_{\xx,y,t}^{\mathfrak{b}} \psi(\xx,y,t)\right\| \leq C\frac{e^{-\frac{|\xx + \alpha t|^2 + y^2}{M(1 + t)}}}{(1+t)\sqrt{1+t}},
\end{align*}
for $\xx,y \in \R, t \geq 0$ and $\mathfrak{b} \in \Z_{\geq 0}^3$ with $1 \leq |\mathfrak{b}| \leq 3$.
\end{theo}

To prove Theorems~\ref{maintheorem} and~\ref{maintheorem2} we decompose the temporal Green's function associated with the parabolic operator $\partial_t - \El$ and we obtain pointwise bounds on each of the components of the Green's function. These pointwise Green's function estimates are given in Section~\ref{sec:pointwise}. In Section~\ref{secitscheme}, we perturb the planar wave-train solution $u(\xx,y,t) = u_\infty(\xx)$ to~\eqref{comove} and establish nonlinear equations for the perturbation and its derivatives. We track the perturbed solution in two different coordinate systems, where one of these coordinates is used to correct for the phase shift of the solution relative to the original wave train. This leads to a closed nonlinear iteration scheme. We apply the pointwise Green's function bounds to this nonlinear iteration scheme in Section~\ref{sec:nonstab} to prove Theorems~\ref{maintheorem} and~\ref{maintheorem2}.

\begin{rem}
The pointwise nonlinear stability analysis of wave-train solutions in one spatial dimension presented in~\cite{JUN,JUNNL} assumes only that perturbations are algebraically localized. More precisely, perturbations $v_0 \in H^2(\R,\R^n)$ need to satisfy only that $|v_0(x)| \leq E_0(1+|x|)^{-r}$ for $x \in \R$ with $0 < E_0 \ll 1$ and $r > 2$. The associated perturbed solution then exhibits both algebraic decay and diffusive Gaussian-like decay. By transferring the estimates in~\cite{JUN,JUNNL} to our setting, we expect that the exponential weights in our main results can be replaced by algebraic weights. However, since the estimates in~\cite{JUN,JUNNL} are technically quite involved and the main point of our analysis is the possibility of nonlocalized perturbations rather than their precise decay properties, we chose to work with exponential weights in this paper for clarity of exposition.
\end{rem}

\section{Green's function decomposition and pointwise estimates} \label{sec:pointwise}

In this section, we show that assumptions (D1)-(D3) on the spectrum of the linearization $\El$ of~\eqref{comove} about the wave train $u(\xx,y,t) = u_\infty(\xx)$ lead to a decomposition of the associated temporal Green's function. We obtain pointwise bounds on each of the components of the Green's function by following the (by now seminal) approach, which was introduced in~\cite{ZUH} and further extended in~\cite{JUN,OHZUM} to the setting of wave trains in one spatial dimension. The decomposition of and the pointwise estimates on the Green's function are the starting point of our nonlinear stability analysis, which is performed in the~\S\ref{sec:nonstab}.

\subsection{Consequences of spectral assumptions} \label{specasscon}

We apply the implicit function theorem to expand the spectrum of $\El$ about the origin. The obtained control on the spectrum is crucial for the decomposition of the temporal Green's function.
\begin{lem} \label{Lem1}
Assume (D1). We complexify ${\pmb \nu}$ and consider the family $L_{\pmb \nu}, {\pmb \nu} \in \C^2$ of operators on $L_\per^2([0,1],\C^n)$ given by~\eqref{Bloch2}. There exists a neighborhood $U \subset \C^2$ of $0$ and an analytic function $\lambda_0 \colon U \to \C$ such that the spectrum of $L_{\pmb \nu}$ in $\Re(\lambda) \geq -\eta$ consists of the simple eigenvalue $\lambda_0({\pmb \nu})$ only. In addition, we have the expansion
\begin{align}\lambda_0({\pmb \nu}) &= i\alpha\nu_x - \theta \nu_x^2 - d_\perp \nu_y^2 + \h({\pmb \nu}), \qquad {\pmb \nu} = (\nu_x,\nu_y) \in U, \label{eigenexp}
\end{align}
with residual $\h \colon U \to \C$ satisfying $|\h({\pmb \nu})| \leq C|{\pmb \nu}|^3$ for some constant $C>0$ and coefficients
\begin{align*}
\alpha &= 2k^2\langle \tilde{u}_\ad, D  \tilde{u}_\infty''\rangle_2 + \omega \in \R, \quad d_\perp = \langle \tilde{u}_\ad, D  \tilde{u}_\infty'\rangle_2 \in \R, \quad \theta \in \R, \end{align*}
where $\tilde{u}_\infty \in \ker(L_0)$ is the restriction of the $1$-periodic wave train $u_\infty$ to $[0,1]$, $\tilde{u}_\ad$ is contained in $\ker(L_0^*)$ such that $\langle \tilde{u}_\ad, \tilde{u}_\infty'\rangle_2 = 1$ and $\langle \cdot,\cdot \rangle_2$ denotes the $L_\per^2([0,1],\C^n)$-inner product.

Moreover, for all ${\pmb \nu} \in U$ the kernels of $L_{\pmb \nu} - \lambda_0({\pmb \nu})$ and its adjoint $\left(L_{\pmb \nu} - \lambda_0({\pmb \nu})\right)^*$ are spanned by analytic eigenfunctions $q \colon U \to L_\per^2([0,1],\C^n)$ and $q_\ad \colon U \to L_\per^2([0,1],\C^n)$, respectively, satisfying $\langle q({\pmb \nu}), q_\ad({\pmb \nu})\rangle_2 = 1$, $q(0) = \tilde{u}_\infty'$ and $q_\ad(0) = \tilde{u}_\ad$. Finally, the associated derivative maps $U \to L^2_\per([0,1],\C^n)$ given by ${\pmb \nu} \mapsto \partial_\xx q({\pmb \nu})$ and ${\pmb \nu} \mapsto \partial_\xx q_\ad({\pmb \nu})$ are also analytic.
\end{lem}
\begin{proof} Since (D1) is satisfied,  it follows from standard perturbation theory~\cite{KAT} that for ${\pmb \nu}$ in some neighborhood $U \subset \C^2$ of $0$ the spectrum of $L_{\pmb \nu}$ in $\Re(\lambda) \geq -\eta$ consists of a simple eigenvalue $\lambda_0({\pmb \nu})$ only, which depends analytically on ${\pmb \nu}$. The fact that the associated eigenvectors $q({\pmb \nu})$ and $q_\ad({\pmb \nu})$ and their derivatives $\partial_\xx q({\pmb \nu})$ and $\partial_\xx q_\ad({\pmb \nu})$ are also analytic in ${\pmb \nu}$ follows by writing the eigenvalue problem $(L_{\pmb \nu} - \lambda) u = 0$ as a first order system $(\partial_\xx - A(\xx,\lambda))\phi = 0$ with $\phi = (u,u_\xx)$ and applying perturbation results from~\cite{KAT}.

Since $\lambda_0({\pmb \nu})$ is algebraically simple and $L_0$ is Fredholm of index $0$ by (D1), the inner product of $q({\pmb \nu})$ and $q_\ad({\pmb \nu})$ cannot vanish by the Fredholm alternative for any ${\pmb \nu} \in U$. Therefore, we can assume without loss of generality that $\langle q({\pmb \nu}), q_\ad({\pmb \nu})\rangle_2 = 1$ and $q(0) = \tilde{u}_\infty'$ for any ${\pmb \nu} \in U$. We define $\tilde{u}_\ad = q_\ad(0)$. The derivative $\partial_{\nu_x} q$ satisfies the equation
\begin{align*}
\left(L_{\pmb \nu} - \lambda_0({\pmb \nu})\right) [\partial_{\nu_x} q]({\pmb \nu}) = [\partial_{\nu_x} \lambda_0]({\pmb \nu}) q({\pmb \nu}) - 2k^2 i D (\partial_\xx + i \nu_x) q({\pmb \nu}) - \omega i q({\pmb \nu}), \quad {\pmb \nu} \in U.
\end{align*}
Taking the inner product with $q_{\ad}({\pmb \nu})$ on both sides and evaluating at ${\pmb \nu} = 0$ yields $[\partial_{\nu_x} \lambda_0](0) = \alpha i$. Similarly, one obtains $[\partial_{\nu_y} \lambda_0](0) = 0$, $[\partial_{\nu_x\nu_y} \lambda_0](0) = 0$ and $[\partial_{\nu_y\nu_y} \lambda_0](0) = -2d_\perp$. Since $\El$ is a real operator, its spectrum in $\C$ is symmetric in the real axis. Therefore, the decomposition~\eqref{specdecomp} yields that $\theta := -[\partial_{\nu_x\nu_x}\lambda_0](0)/2$ must be real. The expansion~\eqref{eigenexp} now follows by analyticity of $\lambda_0 \colon U \to \C$.
\end{proof}
The spectral control about the origin obtained in Lemma~\ref{Lem1} reduces the verification of assumption~(D2) to checking the signs of two Melnikov-type integrals.
\begin{cor}
Assume (D1) is satisfied, then assumption (D2) is satisfied if and only if $d_\perp,\theta > 0$.
\end{cor}
Notice that the linearization of the reaction-diffusion system~\eqref{1D} in one spatial dimension about the wave train solution $u(x,t)=u_\infty(kx-\omega t)$ is given by the operator $\El_0$ -- see~\eqref{Fourier}. Spectral stability of $u(x,t) = u_\infty(kx-\omega t)$ as a solution to~\eqref{1D}, in the sense of~\cite{JONZ,JUN,SAN3}, entails that $0$ is a simple eigenvalue of $L_0$ and that the spectrum of $\El_0$ lies to the left of the imaginary axis, except for a quadratic touching at the origin. This leads to the following result.
\begin{cor} 
Assume the wave train $u(x,t)=u_\infty(kx-\omega t)$ is spectrally stable as solution to~\eqref{1D}, in the sense of~\cite{JONZ,JUN,SAN3}. If it holds $d_\perp > 0$, then (D1)-(D2) are satisfied. Moreover, if $D = \I_n$, then the planar wave train $u(x,y,t)=u_\infty(kx-\omega t)$ is spectrally stable as solution to~\eqref{2D}.
\end{cor}

\subsection{Pointwise Green's function estimates}

We follow~\cite{JUN,OHZUM} and decompose the temporal Green's function associated with the parabolic operator $\partial_t - \El$ in a translational mode and a residual. We obtain pointwise estimates for each of these components of the Green's function.

\begin{theo} \label{theope}
Assume (D1)-(D3) hold true. Let $\chi \colon \R_{\geq 0} \to [0,1]$ be a smooth cut-off function such that $\chi(t) = 0$ for $0 \leq t \leq 1$ and $\chi(t) = 1$ for $t \geq 2$. The temporal Green's function $G(\xx,\xt,y,t)$ associated with the operator $\partial_t - \El$ in~\eqref{lin} can be decomposed as
\begin{align}
\begin{split}
G(\xx,\xt,y,t) &= u_\infty'(\xx)e(\xx,\xt,y,t) + \widetilde{G}(\xx,\xt,y,t),\\
e(\xx,\xt,y,t) &= \frac{\chi(t)}{4 \pi t\sqrt{d_\perp \theta}}e^{-\frac{|\xx-\xt+\alpha t|^2}{4\theta t} - \frac{|y|^2}{4d_\perp t}} u_{\ad} (\xt)^*,
\end{split} \quad \xx,\xt,y \in \R, t \geq 0, \label{decomp}
\end{align}
where $\alpha,\theta$ and $d_\perp$ are as in Lemma~\ref{Lem1} and $u_\ad \colon \R \to \R^n$ is the periodic extension of the eigenfunction $\tilde{u}_\ad \in L^2_\per([0,1],\R^n)$ of the adjoint $L_0^*$ (see also Lemma~\ref{Lem1}). For each $j \in \Z_{\geq 0}$ and multi-indices $\mathfrak{a}, \mathfrak{b} \in \Z_{\geq 0}^3$ with $|\mathfrak{a}| = 1$ and $|\mathfrak{b}| \geq 0$, there exists $C,M > 1$ such that we have the pointwise estimates
\begin{align}
\begin{split}
\left\|\GT(\xx,\xt,y,t)\right\| &\leq Ct^{-1}(1+t)^{-\frac{1}{2}}e^{-\frac{|\xx-\xt + \alpha t|^2 + |y|^2}{Mt}},\\
\left\|D_{\xx,\xt,y}^\mathfrak{a} \GT(\xx,\xt,y,t)\right\| &\leq Ct^{-\frac{3}{2}}e^{-\frac{|\xx-\xt + \alpha t|^2 + |y|^2}{Mt}},\\
\left\|D_{\xx,y,t}^\mathfrak{b} \partial_\xt^j e(\xx,\xt,y,t)\right\| &\leq C(1+t)^{-1-\frac{|\mathfrak{b}|}{2}}e^{-\frac{|\xx-\xt + \alpha t|^2 + |y|^2}{Mt}},
\end{split} \qquad \xx,\xt,y \in \R, t > 0. \label{pointwiseestimates}
\end{align}
\end{theo}

To prove Theorem~\ref{theope} we employ similar methods as in~\cite{JUN}, where pointwise Green's function estimates are obtained for wave-train solutions to reaction-diffusion systems in \emph{one} spatial dimension. To account for an additional spatial direction, we combine the methods in~\cite{JUN} with those from~\cite{HOF}, where pointwise bounds are established for spatially multidimensional viscous shock fronts. Since the proof of Theorem~\ref{theope} follows, by and large, the nontrivial, but by now classical, approach introduced in~\cite{ZUH}, we decided to include the proof of Theorem~\ref{theope} in Appendix~\ref{proofpointwise}.

\section{Nonlinear iteration scheme} \label{secitscheme}

In this section, we perturb the planar wave-train solution $u(\xx,y,t) = u_\infty(\xx)$ to~\eqref{comove} and we establish nonlinear equations for the perturbation and its derivatives. To account for translational invariance and exploit the decomposition~\eqref{decomp} of the Green's function, we introduce a phase function that tracks the shift of the perturbed solution in space relative to the original wave train. Our goal is to obtain a closed nonlinear iteration scheme for the perturbation and the phase function, which will be employed in~\S\ref{sec:nonstab} to prove the nonlinear stability results in~\S\ref{sec:statres}.

\subsection{Perturbation equations}
We consider the perturbed solution
\begin{align*}\ut(\xx,y,t) = u_\infty(\xx) + \vt(\xx,y,t),
\end{align*}
to~\eqref{comove}. The perturbation $\vt$ and its $\xx$-derivatives satisfy
\begin{align}
\left(\partial_t - \El\right)\left[\partial_\xx^i \vt\right] = \NT_i, \qquad i = 0,1,2, \label{pertbeq2}\end{align}
with
\begin{align*}
\NT_0 &:= f(u_\infty+\vt) - f(u_\infty) - f'(u_\infty) \vt, \\ 
\NT_1 &:= \left(f'(u_\infty+\vt)-f'(u_\infty)\right)\left(u_\infty' + \vt_\xx\right),\\
\NT_2 &:= f''(u_\infty + \vt)(u_\infty' + \vt_\xx) \vt_\xx + \left(f''(u_\infty + \vt)\left(u_\infty' + \vt_\xx\right) - f''(u_\infty)u_\infty'\right)u_\infty'\\
&\qquad  + \left(f'(u_\infty+\vt)-f'(u_\infty)\right)\left(u_\infty'' + \vt_{\xx\xx}\right),
\end{align*}
where we suppress the argument $(\xx,y,t)$ of $\vt$ and its derivatives, and the argument $\xx$ of $u_\infty$ and its derivatives for notational convenience. Using Taylor's Theorem, one observes that $\NT_0$ is quadratic in $\vt$, whereas $\NT_1$ and $\NT_2$ contain linear terms in $\vt$ or its derivatives. More precisely, there exist constants $B,C > 1$ such that as long as $\|\vt(t)\|_\infty \leq B$, we have the bounds
\begin{align}
\begin{split}
\left\|\NT_0(\xx,y,t)\right\| \leq &\ C\left\|\vt(\xx,y,t)\right\|^2, \qquad \left\|\NT_1(\xx,y,t)\right\| \leq C\left\|\vt(\xx,y,t)\right\|\left(1+\left\|\vt_\xx(\xx,y,t)\right\|\right),\label{nonlest3}
\end{split}
\end{align}
and, as long as $\|\vt(t)\|_\infty,\|\vt_\xx(t)\|_\infty \leq B$, we have the bound
\begin{align}
\begin{split}
\left\|\NT_2(\xx,y,t)\right\| &\leq C\left(\left\|\vt(\xx,y,t)\right\|\left(1+\left\|\vt_{\xx\xx}(\xx,y,t)\right\|\right) + \left\|\vt_\xx(\xx,y,t)\right\|\right),
\end{split}
\label{nonlest33}
\end{align}
for $\xx,y \in \R$ and $t \geq 0$, where $\|\cdot\|_\infty$ denotes the $L^\infty(\R^2,\R^n)$-norm. To account for translational invariance -- see Remarks~\ref{cancel} and~\ref{remtranslational} -- we introduce another coordinatization for the perturbed solution $\ut$ to~\eqref{comove}. As in~\cite{DSSS}, we write
\begin{align*}\ut(\xx + \psi(\xx,y,t),y,t) = u_\infty(\xx) + v(\xx,y,t),
\end{align*}
where the phase function $\psi \colon \R^2 \times \R_{\geq 0} \to \R$ is to be determined later. By the mean value theorem, there exist constants $B,C > 1$ such that as long as $\|\psi_\xx(t)\|_\infty  \leq B$, we have the bounds
\begin{align}
\begin{split}
\left\|\partial_\xx^i \left(v - \vt\right)(\xx,y,t)\right\| \leq&\ C\sum_{j = 0}^{i} \left(\left\|\partial_{\xx}^{j+1} u_\infty\right\|_{\infty} + \left\|\partial_\xx^{j+1} \vt(t)\right\|_{\infty} \right) \left\|\partial_\xx^{i-j} \psi(\xx,y,t)\right\|, \\
\left\|\left(v_y - \vt_y\right)(\xx,y,t)\right\| \leq&\ \left\|\vt_{\xx y}(t)\right\|_{\infty}\left\|\psi(\xx,y,t)\right\| + \left(\left\|u_\infty'\right\|_{\infty} + \left\|\vt_{\xx}(t)\right\|_{\infty}\right)\left\|\psi_y(\xx,y,t)\right\|,\\
\left\|\left(v_{\xx y} - \vt_{\xx y}\right)(\xx,y,t)\right\| \leq&\ C\left(\left\|\vt_{\xx \xx y}(t)\right\|_{\infty}\left\|\psi(\xx,y,t)\right\| + \left\|\vt_{\xx y}(t)\right\|_{\infty} \left\|\psi_\xx(\xx,y,t)\right\|\right.\\
&\qquad + \left(\left\|u_\infty''\right\|_{\infty} + \left\|\vt_{\xx\xx}(t)\right\|_{\infty} \right) \left\|\psi_y(\xx,y,t)\right\|\\ 
&\qquad \left.+ \left(\left\|u_\infty'\right\|_{\infty} + \left\|\vt_\xx(t)\right\|_{\infty} \right) \left\|\psi_{\xx y}(\xx,y,t)\right\|\right),
\end{split}
\label{conversion}
\end{align}
for $\xx,y \in \R$, $t \geq 0$ and $i = 0,1,2$. Our next step is to derive an equation for the perturbation $v$, the phase $\psi$ and their derivatives. We introduce
\begin{align*} w(\xx,y,t) := \ut(\xx + \psi(\xx,y,t),y,t) = u_\infty(\xx) + v(\xx,y,t).\end{align*}
Substituting $w(\xx,y,t)$ into~\eqref{comove}, while using that $\ut(\xx,y,t)$ solves~\eqref{comove}, leads to the residue induced by the phase
\begin{align*} w_t - D(k^2 w_{\xx\xx} + w_{yy})& - \omega w_\xx - f(w) =\\
&-\ut_3\psi_\xx + \ut_1\psi_t + D\left(\ut_{22}\psi_\xx - \ut_{1 2}\psi_y - (\ut_1 \psi_y)_y - k^2 (\ut_1 \psi_\xx)_\xx\right) + f(\ut)\psi_\xx,\end{align*}
which yields
\begin{align} \begin{split}
\left(\partial_t - \El\right)v =& \ f(u_\infty + v) - f(u_\infty) - f'(u_\infty)v\\
 &- \ut_3\psi_\xx + \ut_1\psi_t + D\left(\ut_{22}\psi_\xx - \ut_{1 2}\psi_y - (\ut_1 \psi_y)_y - k^2 (\ut_1 \psi_\xx)_\xx\right) + f(\ut)\psi_\xx,
\end{split} \label{veq}
\end{align}
where we suppress the arguments $(\xx,y,t)$ of $v,w,\psi$ and their derivatives, the arguments $(\xx+\psi(\xx,y,t),y,t)$ of $\ut$ and its derivatives and the arguments $\xx$ of $u_\infty$ and its derivatives. To avoid confusion by suppressing the arguments of $\ut$, we denote the partial derivative of $\ut$ with respect to its $j$-th argument by $\ut_j$. In order to obtain a closed system in $v$ and $\psi$, we need to eliminate the $\ut$-terms from the right-hand side of~\eqref{veq}. Therefore, we express $\ut_1$, $\ut_2$ and $\ut_3$ in terms of $v$, $\psi$ and their derivatives as follows. First, we calculate
\begin{align*}
u_\infty' + v_\xx = w_\xx = \ut_1(1+\psi_\xx), \quad v_t = w_t = \ut_3 + \ut_1\psi_t, \quad v_y = w_y = \ut_2 + \ut_1\psi_y,
\end{align*}
yielding the identities
\begin{align}
\ut_1 &= \frac{1}{1+\psi_\xx}\left(u_\infty' + v_\xx\right) \quad \ut_2 = v_t - \frac{\psi_t}{1+\psi_\xx}\left(u_\infty' + v_\xx\right), \quad \ut_3 = v_y - \frac{\psi_y}{1+\psi_\xx}\left(u_\infty' + v_\xx\right),
\label{exprveq}
\end{align}
where we suppress the arguments again. By employing the identities
\begin{align*}
\left(\partial_t - \El\right)\left[u_\infty' \psi\right] =&\ \El[u_\infty'] \psi - D\left(k^2 (u_\infty' \psi_\xx)_\xx + u_\infty' \psi_{yy}\right) - \left(D k^2 u_\infty'' + \omega u_\infty'\right) \psi_\xx + u_\infty' \psi_t\\
=&\ - D\left(k^2 (u_\infty' \psi_\xx)_\xx + u_\infty' \psi_{yy}\right) + f(u_\infty)\psi_\xx + u_\infty' \psi_t,\\
\left(\partial_t - \El\right)\left[v\psi_\xx\right] =&\ v_t \psi_\xx + v\psi_{\xx t} - D\left(k^2\left(v_{\xx\xx} \psi_\xx + 2v_\xx\psi_{\xx\xx} + v \psi_{\xx\xx\xx}\right) + v_{yy} \psi_\xx + 2v_y\psi_{\xx y} + v \psi_{\xx yy}\right)\\
& \qquad - \omega\left(v_\xx \psi_\xx + v\psi_{\xx\xx}\right) - f'(u_\infty) v \psi_\xx,
\end{align*}
and substituting~\eqref{exprveq} into the right hand side of~\eqref{veq}, we obtain the perturbation equation
\begin{align}
\left(\partial_t - \El\right)\left[v - u_\infty' \psi\right] = \Non - \left(\partial_t - \El\right)\left[v\psi_\xx\right], \label{pertbeq}
\end{align}
where the nonlinearity $\Non$ is given by
\begin{align}
\begin{split}
\Non :=& \left(f(u_\infty+v) - f(u_\infty) - f'(u_\infty) v\right)\left(1+\psi_\xx\right) + v_\xx\psi_t + v\psi_{\xx t} - \omega\left(v_\xx \psi_\xx + v \psi_{\xx\xx}\right)\\
    &\qquad - D\left(v\left(k^2\psi_{\xx\xx\xx} + \psi_{\xx yy}\right) + v_\xx\left(3k^2 \psi_{\xx\xx} + \psi_{yy}\right) + 2v_y\psi_{\xx y} + 2k^2 v_{\xx\xx} \psi_\xx + 2v_{\xx y}\psi_y\right)\\
    &\qquad + \frac{D}{1+\psi_\xx} \left(2\left(u_\infty' + v_\xx\right)\left(\psi_y\psi_{\xx y} + k^2 \psi_\xx \psi_{\xx\xx}\right) + \left(u_\infty'' + v_{\xx\xx}\right)\left(\psi_y^2 + k^2 \psi_\xx^2\right)\right)\\
    &\qquad - \frac{D}{\left(1+\psi_\xx\right)^2} \left(u_\infty' + v_\xx\right)\left(\psi_y^2 + k^2 \psi_\xx^2\right)\psi_{\xx\xx}.
\end{split} \label{nonlinearity}
\end{align}
Using Taylor's Theorem, it is relatively straightforward to observe that the nonlinearity~\eqref{nonlinearity} is quadratic in $v$, $\psi$ and their derivatives. More precisely, there exist constants $B,C > 1$ such that as long as it holds
\begin{align*}
\|v(t)\|_\infty + \sum_{|\mathfrak{a}| = 1} \left\|D_{\xx,y}^{\mathfrak{a}} \psi(t)\right\|_\infty \leq B,
\end{align*}
we have the bound
\begin{align}
\begin{split}
\left\|\Non(\xx,y,t)\right\| \leq \ &C\left[\left(\left\|v(\xx,y,t)\right\| + \sum_{1 \leq |\mathfrak{a}| \leq 3} \left\|D_{\xx,y,t}^{\mathfrak{a}} \psi(\xx,y,t)\right\|\right)^2\right.\\
&\left.\qquad + \left(\left\|v_y(\xx,y,t)\right\| + \sum_{0 \leq |\mathfrak{b}| \leq 1} \left\|D_{\xx,y}^{\mathfrak{b}} v_\xx(\xx,y,t)\right\|\right) \sum_{1 \leq |\mathfrak{c}| \leq 2} \left\|D_{\xx,y,t}^{\mathfrak{c}} \psi(\xx,y,t)\right\|\right],\\
\end{split}
\label{nonlest1}
\end{align}
for $\xx,y \in \R$ and $t \geq 0$.

\begin{rem} \label{cancel}
In the perturbation equation~\eqref{pertbeq} for $v$, we grouped terms that are nonlinear in $(v,\psi)$ and their derivatives on the right-hand side, whereas the left-hand side contains all contributions that are linear in $(v,\psi)$ and their derivatives. This decomposition suggests new coordinates $(w,\psi)$ with $w = v + u_\infty'\psi$. The advantage of working in these coordinates is that, by choosing $\psi$ appropriately, the contribution $u_\infty' \psi$ accounts for the translational mode $u_\infty'(\xx)e(\xx,\xt,y,t)$ of the Green's function in the nonlinear iteration -- see Remark~\ref{remtranslational}. The principle of canceling the slowly decaying translational mode via these type of coordinatizations has been developed independently in~\cite{JONZV,JONZ} and~\cite{DSSS,SAN3} using different methods -- see~\cite[Remark~5.1]{JONZNL} for a review.

In the nonlinear right-hand side of~\eqref{pertbeq}, we introduced the term $\left(\partial_t - \El\right)\left[v\psi_\xx\right]$ in order to eliminate any temporal derivatives of $v$ in the nonlinear iteration. Indeed, one observes from~\eqref{nonlinearity} that the residual nonlinearity $\Non$ contains only $\xx$- and $y$-derivatives of $v$. In Remark~\ref{motvt} below, we will explain how we control spatial derivatives of $v$ in the nonlinear iteration.
\end{rem}

\subsection{Duhamel's integral formulation}
Our goal is obtain a closed nonlinear iteration scheme by integrating~\eqref{pertbeq}, while exploiting the decomposition~\eqref{decomp} of the Green's function. Applying Duhamel's formula to~\eqref{pertbeq} yields the equivalent integral equation
\begin{align}
\begin{split}
v(\xx,y,t) =& \ u_\infty'(\xx)\psi(\xx,y,t) - v(\xx,y,t)\psi_\xx(\xx,y,t) + \int_{\R^2} G(\xx,\xt,y-\yt,t) v_0(\xt,\yt) d\xt d\yt\\
&\qquad + \int_0^t \int_{\R^2} G(\xx,\xt,y-\yt,t-s) \Non(\xt,\yt,s)d\xt d\yt ds,\\
\end{split} \label{vintegral}
\end{align}
with $\xx,y \in \R, t \geq 0$ and $v_0(\xx,y) = v(\xx,y,0)$. Now we define the phase function $\psi$ as the solution to the integral equation
\begin{align}
\begin{split}
\psi(\xx,y,t) =&  -\int_{\R^2} e(\xx,\xt,y-\yt,t) v_0(\xt,\yt) d\xt d\yt\\
&\qquad - \int_0^t \int_{\R^2} e(\xx,\xt,y-\yt,t-s) \Non(\xt,\yt,s)d\xt d\yt ds.
\end{split}
\label{defpsi}
\end{align}
Since it holds $e(\xx,\xt,y,0) = 0$ for all $\xx,\xt,y \in \R$ by construction, we have thus taken $\psi(\xx,y,0) = 0$ as initial condition for $\psi$. This choice of $\psi$ leads via~\eqref{decomp} to the system of integral equations
\begin{align}
\begin{split}
D_{\xx,y}^{\mathfrak{a}} v(\xx,y,t) =& \ -D_{\xx,y}^{\mathfrak{a}} v(\xx,y,t) \psi_\xx(\xx,y,t) + \int_{\R^2} D_{\xx,y}^{\mathfrak{a}} \GT(\xx,\xt,y-\yt,t) v_0(\xt,\yt) d\xt d\yt\\
&\qquad + \int_0^t \int_{\R^2} D_{\xx,y}^{\mathfrak{a}} \GT(\xx,\xt,y-\yt,t-s) \Non(\xt,\yt,s)d\xt d\yt ds,\\
D_{\xx,y,t}^{\mathfrak{b}} \psi(\xx,y,t) =& -\int_{\R^2} D_{\xx,y,t}^{\mathfrak{b}} e(\xx,\xt,y-\yt,t) v_0(\xt,\yt) d\xt d\yt\\
&\qquad- \int_0^t \int_{\R^2} D_{\xx,y,t}^{\mathfrak{b}} e(\xx,\xt,y-\yt,t-s) \Non(\xt,\yt,s)d\xt d\yt ds,
\end{split}
\label{integralscheme}
\end{align}
with $\xx,y \in \R, t \geq 0$, and multi-indices $\mathfrak{a} \in \Z_{\geq 0}^2$ and $\mathfrak{b} \in \Z_{\geq 0}^3$ satisfying $0 \leq |\mathfrak{a}|\leq 1$ and $0 \leq |\mathfrak{b}| \leq 3$, where we use that $e(\xx,\xt,y,0) = 0$ for all $\xx,\xt,y \in \R$ to determine the $t$-derivative of the second integral in~\eqref{defpsi}.

To obtain a closed nonlinear iteration scheme from~\eqref{integralscheme}, we still need to control the second derivatives $v_{\xx\xx}$ and $v_{\xx y}$ in the nonlinearity~\eqref{nonlinearity}. We emphasize that we cannot obtain such control by differentiating the $v_\xx$- and $v_y$-equation in~\eqref{integralscheme} with respect to $\xx$, since, as explained in the subsequent Remark~\ref{motvt}, the second spatial derivatives of the Green's function $\GT(\xx,\xt,y,t)$ do not permit $t$-integrable bounds. Instead, we proceed as follows. First, we eliminate the second derivatives $v_{\xx\xx}$ and $v_{\xx y}$ in the $v$-equation in~\eqref{integralscheme} using integration by parts at the expense of taking more derivatives of the $\psi$-dependent terms. We note that this elimination is only possible due to the fact that all derivatives of $v$ in~\eqref{nonlinearity} are paired with $\psi$-terms and not with other $v$-terms, and due to the fact that the pointwise bounds on the Green's functions $\partial_{\xt} \GT(\xx,\xt,y,t)$ and $\partial_{y} \GT(\xx,\xt,y,t)$, obtained in Theorem~\ref{theope}, are integrable over $t$. Since we control also higher order $\xt$- and $y$-derivatives of the Green's function $e(\xx,\xt,y,t)$ via Theorem~\ref{theope}, we can apply integration by parts multiple times in the $\psi$-equation to eliminate \emph{all} derivatives of $v$ from the nonlinearity. The elimination of (second) derivatives of $v$ from the nonlinearities in the $v$- and $\psi$-equations in~\eqref{integralscheme} lead to sharp estimates, which are necessary to close the nonlinear iteration. However, the second derivatives of $v$ cannot be eliminated from the nonlinearity in the $v_\xx$- and $v_y$-equations in~\eqref{integralscheme} using integration by parts, since this would again lead to second spatial derivatives of the Green's function $\GT(\xx,\xt,y,t)$, which do not permit $t$-integrable bounds. To control the second derivatives $v_{\xx\xx}$ and $v_{\xx y}$ of $v$, we will append integral equations for the derivatives $\vt_{\xx\xx},\vt_{\xx y}, \vt_{\xx \xx \xx}$ and $\vt_{\xx \xx y}$ of the `unshifted' perturbation $\vt$ to the scheme.

Thus, exploiting that $\GT(\xx,\xt,y,t)$, $e(\xx,\xt,y,t)$ and its derivatives are exponentially localized in space and assuming that there exists a bound $B > 1$ such that
\begin{align*}
\sup_{0 \leq s \leq t} \left(\|v(s)\|_\infty + \|v_\xx(s)\|_\infty + \sum_{1 \leq |\mathfrak{b}| \leq 3} \left\|D_{\xx,y}^{\mathfrak{b}} \psi(s)\right\|_\infty\right) \leq B,
\end{align*}
we integrate by parts in~\eqref{integralscheme} to obtain the following equivalent integral system
\begin{align}
\begin{split}
v(\xx,y,t) =& \ -v(\xx,y,t) \psi_\xx(\xx,y,t) + \int_{\R^2} \GT(\xx,\xt,y-\yt,t) v_0(\xt,\yt) d\xt d\yt\\
            & \qquad + \int_0^t \int_{\R^2} \GT(\xx,\xt,y-\yt,t-s) \Non_1(\xt,\yt,s)d\xt d\yt ds \\
            & \qquad + \int_0^t \int_{\R^2} \partial_{\xt} \GT(\xx,\xt,y-\yt,t-s) \Non_2(\xt,\yt,s)d\xt d\yt ds\\
            & \qquad + \int_0^t \int_{\R^2} \partial_y \GT(\xx,\xt,y-\yt,t-s) \Non_3(\xt,\yt,s)d\xt d\yt ds\\
D_{\xx,y}^{\mathfrak{c}} v(\xx,y,t) =& \ -D_{\xx,y}^{\mathfrak{c}} \left[v(\xx,y,t) \psi_\xx(\xx,y,t)\right] + \int_{\R^2} D_{\xx,y}^{\mathfrak{c}} \GT(\xx,\xt,y-\yt,t) v_0(\xt,\yt) d\xt d\yt\\
&\qquad + \int_0^t \int_{\R^2} D_{\xx,y}^{\mathfrak{c}} \GT(\xx,\xt,y-\yt,t-s) \Non(\xt,\yt,s)d\xt d\yt ds,\\
D_{\xx,y,t}^{\mathfrak{b}} \psi(\xx,y,t) =& -\int_{\R^2} D_{\xx,y,t}^{\mathfrak{b}} e(\xx,\xt,y-\yt,t) v_0(\xt,\yt) d\xt d\yt \\
&\qquad + \sum_{0 \leq |\mathfrak{a}| \leq 2} \int_0^t \int_{\R^2} D_{\xx,y,t}^{\mathfrak{b}} D_{\xt,y}^{\mathfrak{a}} e(\xx,\xt,y-\yt,t-s) \NTT_{\mathfrak{a}}(\xt,\yt,s)d\xt d\yt ds,
\end{split} \label{integralscheme2}
\end{align}
with $\xx,y \in \R, t \geq 0$ and multi-indices $\mathfrak{c} \in \Z_{\geq 0}^2$ and $\mathfrak{b} \in \Z_{\geq 0}^3$ satisfying $|\mathfrak{c}|=1$ and $0 \leq |\mathfrak{b}| \leq 3$, where the nonlinearity $\Non_1$ contains the term
\begin{align*} Dv_\xx \left(2\psi_{y y} + \left(2k^2\psi_\xx - \frac{\psi_y^2 + k^2\psi_\xx^2}{1+\psi_\xx}\right)_\xx\right),\end{align*}
with a derivative of $v$, whereas all other terms present in $\Non_1$ do not depend on derivatives of $v$. Similarly, the nonlinearity $\Non_2$ only contains the term
\begin{align*} D v_\xx \left(2k^2\psi_\xx - \frac{\psi_y^2 + k^2\psi_\xx^2}{1+\psi_\xx}\right),\end{align*}
with a derivative of $v$, the nonlinearity $\Non_3$ only contains the term
\begin{align*} -2D v_\xx \psi_y,\end{align*}
with a derivative of $v$ and the nonlinearities $\NTT_{\mathfrak{a}}$ contain no derivatives of $v$ for all multi-indices $\mathfrak{a} \in \Z_{\geq 0}^2$ with $0 \leq |\mathfrak{a}| \leq 2$. Using Taylor's Theorem, one readily observes that the obtained nonlinearities $\Non_j,\NTT_{\mathfrak{a}}$ are quadratic in $v$, $\psi$ and their derivatives. More precisely, there exist constants $B,C > 1$ such that as long as
\begin{align*}
\|v(t)\|_\infty + \sum_{1 \leq |\mathfrak{b}| \leq 2} \left\|D_{\xx,y}^{\mathfrak{b}} \psi(t)\right\|_\infty \leq B,
\end{align*}
we have the bounds
\begin{align}
\begin{split}
\left\|\Non_j(\xx,y,t)\right\| \leq&\ C\left[\left(\|v(\xx,y,t)\| + \sum_{1 \leq |\mathfrak{b}| \leq 3} \left\|D_{\xx,y,t}^{\mathfrak{b}} \psi(\xx,y,t)\right\|\right)^2\right. \\
&\left.\qquad \qquad \qquad \qquad \qquad \qquad + \left\|v_\xx(\xx,y,t)\right\| \sum_{1 \leq |\mathfrak{c}| \leq 2} \left\|D_{\xx,y}^{\mathfrak{c}} \psi(\xx,y,t)\right\|\right],\\
\left\|\NTT_{\mathfrak{a}}(\xx,y,t)\right\| \leq&\ C\left(\|v(\xx,y,t)\| + \sum_{1 \leq |\mathfrak{b}| \leq 3} \left\|D_{\xx,y,t}^{\mathfrak{b}} \psi(\xx,y,t)\right\|\right)^2,
\end{split} \label{nonlest2}
\end{align}
for $\xx,y \in \R$, $t \geq 0$, $j = 1,2,3$ and $\mathfrak{a} \in \Z_{\geq 0}^2$ with $0 \leq |\mathfrak{a}| \leq 2$. Comparing the estimates on $\Non_j, j = 1,2,3$ to the estimate~\eqref{nonlest1} on $\Non$, we observe that we have indeed gained one derivative of $v$. Since the bounds on the first spatial derivatives of $v$, obtained in the subsequent analysis in~\S\ref{sec:nonstab}, exhibit stronger decay than those on the second spatial derivatives of $v$, this will lead to sharp estimates on $v$, and on $\psi$ and its derivatives.

However, as mentioned before, the integral equations~\eqref{integralscheme2} do not give rise to a closed nonlinear iteration scheme yet, because we still need to control the second derivatives $v_{\xx\xx}$ and $v_{\xx y}$ occurring in the nonlinearity $\Non$ in the $v_\xx$- and $v_y$-equations in~\eqref{integralscheme2}. This can be achieved by appending integral equations for the derivatives $\vt_{\xx\xx},\vt_{\xx y}, \vt_{\xx \xx \xx}$ and $\vt_{\xx \xx y}$ of the `unshifted' perturbation $\vt$ to the scheme~\eqref{integralscheme2} and employing the estimates~\eqref{conversion} -- see also Remark~\ref{motvt}. Applying Duhamel's formula to~\eqref{pertbeq2} yields
\begin{align}
\begin{split}
D_{\xx,y}^{\mathfrak{a}} \partial_{\xx}^i \vt(\xx,y,t) =&\ \int_{\R^2} D_{\xx,y}^{\mathfrak{a}} G(\xx,\xt,y-\yt,t) \partial_{\xt}^i v_0(\xt,\yt) d\xt d\yt\\
&\qquad + \int_0^t \int_{\R^2} D_{\xx,y}^{\mathfrak{a}} G(\xx,\xt,y-\yt,t-s) \NT_i(\xt,\yt,s)d\xt d\yt ds,
\end{split} \label{integralscheme3}
\end{align}
with $\xx,y \in \R$, $t \geq 0$, $i = 0,1,2$ and $\mathfrak{a} \in \Z_{\geq 0}^2$ satisfying $0 \leq |\mathfrak{a}| \leq 1$, where we used that $\psi(\xx,y,0) = 0$ implies $v_0(\xx,y) = v(\xx,y,0) = \vt(\xx,y,0)$ for all $\xx, y \in \R$. Observe that by~\eqref{nonlest3} and~\eqref{nonlest33} the nonlinearities $\NT_i, i = 0,1,2,$ can be estimated in terms of $\vt$, $\vt_\xx$ and $\vt_{\xx \xx}$. Thus, combining the integral equations~\eqref{integralscheme2} and~\eqref{integralscheme3} and the estimates~\eqref{nonlest3},~\eqref{nonlest33},~\eqref{conversion},~\eqref{nonlest1} and~\eqref{nonlest2} yields a closed nonlinear iteration scheme.

\begin{rem}\label{remtranslational}
Applying Duhamel's formula to~\eqref{pertbeq2} gives the integral equation
\begin{align}
\begin{split}
\vt(\xx,y,t) =&\ \int_{\R^2} G(\xx,\xt,y-\yt,t) v_0(\xt,\yt) d\xt d\yt\\
&\qquad + \int_0^t \int_{\R^2} G(\xx,\xt,y-\yt,t-s) \NT_0(\xt,\yt,s)d\xt d\yt ds, \end{split} \label{integralscheme4}
\end{align}
with $\xx,y \in \R$ and $t \geq 0$. Observe that~\eqref{integralscheme4} yields a closed nonlinear iteration scheme, since the nonlinearity $\NT_0$ depends only on $\vt$ itself. However, because $\NT_0$ is quadratic in $\vt$ in general, the Green's function estimate
\begin{align*} \left\|G(\xx,\xt,y,t)\right\| \leq Ct^{-1}e^{-\frac{|\xx-\xt + \alpha t|^2 + |y|^2}{Mt}}, \qquad \xx,\xt,y \in \R, t > 0,\end{align*}
obtained in Theorem~\ref{theope} is not strong enough to close a nonlinear iteration argument. The decomposition in Theorem~\ref{theope} factors out the translational mode $u_\infty'(\xx)e(\xx,\xt,y,t)$ of the Green's function, which yields an additional decay factor $t^{-\frac{1}{2}}$ in the estimate of the residual. Moreover, the introduction of the phase function $\psi(\xx,y,t)$ in the nonlinear iteration scheme leads to a contribution $u_\infty'\psi$ in~\eqref{vintegral}, which cancels the translational mode of the Green's function by a judicious choice~\eqref{defpsi} of $\psi$. The additional decay of the residual mode of the Green's function can then be exploited to close the nonlinear iteration argument. The trade-off is that the introduction of the phase function $\psi(\xx,y,t)$ yields derivatives of $v$ in the nonlinearity, which needs to be controlled -- see also Remark~\ref{motvt}.
\end{rem}

\begin{rem} \label{motvt}
In order to obtain a closed nonlinear iteration scheme from~\eqref{integralscheme}, one needs to control the second derivatives $v_{\xx y}$ and $v_{\xx\xx}$ in the nonlinearity~\eqref{nonlinearity}. Naively, one might introduce integral equations for these second derivatives of $v$ by differentiating~\eqref{integralscheme}. However, the second order derivatives of the Green's function satisfy the bounds
\begin{align*} \left\|D_{\xx,\xt,y}^{\mathfrak{b}} \GT(\xx,\xt,y,t)\right\| \leq Ct^{-\frac{3}{2}}\left(1+t^{-\frac{1}{2}}\right)e^{-\frac{|\xx-\xt + \alpha t|^2 + |y|^2}{Mt}}, \qquad \xx,\xt,y \in \R, t > 0, \mathfrak{b} \in \Z_{\geq 0}^3, |\mathfrak{b}| = 2,\end{align*}
for some constants $C,M > 1$, which leads, after integration over $\xt$ and $y$, to the nonintegrable factor $(t-s)^{-1}$ in the estimation of the integral equations for $v_{\xx y}$ and $v_{\xx\xx}$. Integrating by parts to move a derivative from the Green's function to the nonlinearity, introduces third derivatives of $v$ in the nonlinearity and thus transfers the problem of controlling second derivatives of $v$ to controlling third derivatives of $v$.

The same problem occurs in the nonlinear stability analyses~\cite{JONZNL,JONZ,JUN,JUNNL} of wave-train solutions to reaction-diffusion systems in one spatial dimension. There one establishes a nonlinear damping estimate that controls the $H^k$-norm of the perturbation $v$ for $k \geq 0$ by its $L^2$-norm and by the $H^k$-norm of $(\psi_\xx,\psi_t)$. This yields nonlinear $L^1 \cap H^k \to H^k$-stability in~\cite{JONZ} and nonlinear $L^1 \cap H^k \to L^\infty$-stability with nonlocal phase in~\cite{JONZNL}. The damping estimate from~\cite{JONZNL,JONZ} is also employed in the pointwise nonlinear stability analyses in~\cite{JUN,JUNNL}. Since the derivatives of $v$ are paired with $\psi$-terms in the nonlinearity, one only needs $L^\infty$-bounds on the derivatives of $v$ in~\cite{JUN,JUNNL}. These $L^\infty$-bounds follow via Sobolev interpolation from  $H^k$-bounds on $v$ that are obtained using the damping estimate from~\cite{JONZNL,JONZ}.

One can establish a $2$-dimensional equivalent of the damping estimate in~\cite{JONZNL,JONZ}, which bounds the $H^k(\R^2,\R^n)$-norm of $v$ in terms of its $L^2(\R^2,\R^n)$-norm and the $H^k(\R^2,\R^n)$-norm of $(\psi_\xx,\psi_t)$. However, such a damping estimate cannot be employed in our analysis, since we allow for nonlocal perturbations. More precisely, we consider planar perturbations that are not in $L^2(\R^2,\R^n)$, whereas in~\cite{JONZNL,JONZ,JUN,JUNNL} the perturbations are $L^2$-localized on the real line (allowing possibly for a nonlocal phase modulation). Since we consider perturbations that are only nonlocalized in one spatial direction (say the $y$-direction), it may be possible to adapt the damping estimate in~\cite{JONZNL,JONZ} to work in the mixed $L^\mathfrak{p}$-space $L^{\mathfrak{p}}(\R^2,\R^n)$, where $\mathfrak{p}$ is the tuple $(2,\infty)$, endowed with the norm
\begin{align*} \|u\|_{\mathfrak{p}} = \sup_{y \in \R} \left(\int_\R \left\|u(\xx,y)\right\|^2 d\xx\right)^{\frac{1}{2}}. \end{align*}
However, we expect that one needs to append a damping estimate of $v_y$ in $L^2(\R^2,\R^n)$ in order to control the $y$-derivatives in the damping estimate of $v$ in $L^{\mathfrak{p}}(\R^2,\R^n)$, thus requiring the derivatives of $v$ to be localized in $\R^2$.

In this paper, we choose an alternative approach by appending the integral equations~\eqref{integralscheme3} for the derivatives of the `unshifted' perturbation $\vt$ to the nonlinear iteration scheme~\eqref{integralscheme2}. All derivatives in the nonlinearity~\eqref{nonlinearity} vanish, when the phase function $\psi$ is set to $0$. Therefore, the spatial derivatives of $\vt$ can be controlled without running into problems due to nonintegrable Green's function bounds. Then, estimates like~\eqref{conversion} can be used to control the derivatives of $v$. The advantage over a possible $L^{\mathfrak{p}}(\R^2,\R^n)$-damping estimate is that we do not require localization of any derivative of $v$ in $\R^2$. The trade-off is that the obtained estimates on $v_{\xx\xx}$ and $v_{\xx y}$ are not as strong as one might expect due to the fact that we do not correct for a phase shift relative to the original wave train while estimating $\vt$ and its derivatives. However, by \emph{combining} these estimates with the ones on the integral scheme~\eqref{integralscheme2} which does track the phase shift, they are strong enough to close the nonlinear iteration argument -- see Remark~\ref{remtranslational}. Moreover, for fully exponentially localized perturbations, our method yields the same optimal decay bounds as the ones one obtains with a damping estimate in $H^k(\R^2,\R^n)$ up to a possibly artificial $\log(2+t)$-factor -- see Theorem~\ref{maintheorem2} and Remark~\ref{optimality}.
\end{rem}

\section{Nonlinear stability analysis} \label{sec:nonstab}

We prove the nonlinear stability results stated in~\S\ref{sec:statres} by applying the pointwise Green's function estimates from Theorem~\ref{theope} to the nonlinear iteration scheme consisting of~\eqref{conversion},~\eqref{integralscheme2} and~\eqref{integralscheme3}. We start with the proof of Theorem~\ref{maintheorem} concerning nonlocalized perturbations.

\begin{proof}[Proof of Theorem~\ref{maintheorem}] In this proof, the constant $M > 1$ is as in Theorem~\ref{theope}, whereas $C>1$ denotes a constant, which is independent of $E_0,\xx,\xt,y$ and $t$, that will be taken larger if necessary.

Short-time existence theory for semilinear parabolic equations -- see~\cite{HEN,LUN} -- yields via a standard contraction-mapping argument that there exists a maximal $T_0 \in (0,\infty]$ such that~\eqref{pertbeq2} has a mild solution $\vt(t)$ in $W^{2,\infty}(\R^2,\R^n)$ on $[0,T_0)$ with $\vt(0) = v_0$. The map $t \mapsto \|\vt(t)\|_{W^{2,\infty}(\R^2,\R^n)}$ is continuous on $[0,T_0)$ and, if $T_0$ is finite, it must blow up as $t \uparrow T_0$. By a similar contraction-mapping argument, the semilinear parabolic system given by~\eqref{pertbeq2},~\eqref{pertbeq} and~\eqref{defpsi} -- and hence the integral system given by~\eqref{integralscheme2} and~\eqref{integralscheme3} -- has a solution $(v(t),\psi(t))$ in $X_\infty := W^{2,\infty}(\R^2,\R^n) \times W^{3,\infty}(\R^2,\R)$ on some maximal time interval $[0,T_1), T_1 \in (0,\infty]$ with initial condition $(v(0),\psi(0)) = (v_0,0)$ such that the template function
\begin{align*}
\eta(t) &:= \sup_{\begin{smallmatrix} 0 \leq s \leq t\\ \xx,y \in \R\end{smallmatrix}} \left(e^{\frac{|\beta \xx + \gamma y + \alpha \beta s|^2}{M(1 + s)}}  \left[(1+s)\left(\frac{\left\|v(\xx,y,s)\right\|}{\log(2+s)} + \sum_{1 \leq |\mathfrak{b}| \leq 3} \left\|D_{\xx,y,s}^\mathfrak{b} \psi(\xx,y,s)\right\|\right)\right.\right.\\
&\qquad\qquad\qquad \left.\left.\phantom{\sum_{1 \leq |\mathfrak{b}| \leq 3}} + \sqrt{s}\left(\left\|v_\xx(\xx,y,s)\right\|+ \left\|v_y(\xx,y,s)\right\| + \left\|\vt_y(\xx,y,s)\right\|\right)\right.\right.\\
&\qquad\qquad\qquad \left.\left.\phantom{\sum_{1 \leq |\mathfrak{b}| \leq 3}} + \sqrt{1+s}\left(\left\|\vt(\xx,y,s)\right\| + \left\|\vt_\xx(\xx,y,s)\right\| +\left\|\psi(\xx,y,s)\right\|\right)\right]\right.\\
&\qquad \left.\phantom{\sum_{1 \leq |\mathfrak{b}| \leq 3}} +
\frac{\sqrt{s}}{1+\sqrt{s}}\left(\left\|v_{\xx\xx}(\xx,y,s)\right\| + \left\|v_{\xx y}(\xx,y,s)\right\| + \left\|\vt_{\xx\xx}(\xx,y,s)\right\| + \left\|\vt_{\xx y}(\xx,y,s)\right\|\right)\right),
\end{align*}
is well-defined and continuous on $[0,T_1)$ and, if $T_1$ is finite, then $\eta(t)$ must blow up as $t \uparrow T_1$. Our goal is to prove that there exist constants $B > 0$ and $C > 1$ such that for all $t \geq 0$ with $\eta(t) \leq B$ we have
\begin{align} \eta(t) \leq C\left(E_0 + \eta(t)^2\right). \label{etaest}\end{align}
Since $\eta$ must be continuous as long as it remains small, we can apply continuous induction using estimate~\eqref{etaest}. Thus, provided that $E_0 < \frac{1}{4C}$, it follows $\eta(t) \leq 2CE_0$ for \emph{all} $t \geq 0$, which readily yields the result.

We prove the key estimate~\eqref{etaest}. We employ two integral identities
\begin{align}
\begin{split}
\int_{\R^2} e^{-\frac{|\xx-\xt + \alpha t|^2 + |y-\yt|^2}{Mt} - \frac{|\beta \xt + \gamma \yt|^2}{M}}d\xt d\yt &= \frac{M\pi t e^{-\frac{|\beta \xx + \gamma y + \alpha \beta t|^2}{M(1 + t)}}}{\sqrt{1+t}},\\
\int_{\R^2} e^{-\frac{|\xx-\xt + \alpha (t-s)|^2 + |y-\yt|^2}{M(t-s)} - \frac{|\beta \xt + \gamma \yt + \alpha \beta s|^2}{M(1+s)}}d\xt d\yt &= \frac{M\pi (t-s) \sqrt{1+s} e^{-\frac{|\beta \xx + \gamma y + \alpha \beta t|^2}{M(1 + t)}}}{\sqrt{1+t}},
\end{split} \label{intid4}
\end{align}
with $\xx,y,s,t \in \R$ satisfying $0 \leq s < t$. These identities are obtained by adding the fractions in the exponent of the integrand, while recalling that $(\beta,\gamma) \in \R^2$ is a unit vector, and applying the (standard) formula
\begin{align} \int_{\R} e^{az(z-2b)}dz &= \frac{\sqrt{\pi}}{\sqrt{|a|}}e^{-ab^2}, &\quad a < 0, b \in \C, \label{intid} \end{align}
twice to evaluate the double integral.

Let $B > 1$ be as in~\S\ref{secitscheme} and assume $t > 0$ is such that $\eta(t) \leq B$. We start by bounding the solutions to the integral system~\eqref{integralscheme2} in terms of $\eta(t)$. By estimate~\eqref{nonlest2} the nonlinearities in~\eqref{integralscheme2} can be bounded as
\begin{align}
\begin{split}
\left\|\Non_j(\xt,\yt,s)\right\| &\leq C\frac{\eta(t)^2 e^{-\frac{|\beta \xt + \gamma \yt + \alpha \beta s|^2}{M(1+s)}}}{(1+s)\sqrt{s}},
\qquad \qquad \left\|\Non(\xt,\yt,s)\right\| \leq C\frac{\eta(t)^2 e^{-\frac{|\beta \xt + \gamma \yt + \alpha \beta s|^2}{M(1+s)}}}{\sqrt{1+s}\sqrt{s}},\\
&\left\|\NTT_{\mathfrak{a}}(\xt,\yt,s)\right\| \leq C\frac{\eta(t)^2 \left(\log(2+s)\right)^2 e^{-\frac{|\beta \xt + \gamma \yt + \alpha \beta s|^2}{M(1+s)}}}{(1+s)^2},
\end{split} \label{nonlinest}
\end{align}
with $\xt,\yt \in \R$, $0 < s \leq t$, $j = 1,2,3$ and $\mathfrak{a} \in \Z_{\geq 0}^3$ satisfying $0 \leq |\mathfrak{a}| \leq 2$ . Thus, applying the pointwise Green's function estimates in Theorem~\ref{theope} to the integral system~\eqref{integralscheme2} and employing~\eqref{intid4} and~\eqref{nonlinest}, yields
\begin{align}
\begin{split}
\left\|v(\xx,y,t)\right\| &\leq C \frac{e^{-\frac{|\beta \xx + \gamma y + \alpha \beta t|^2}{M(1 + t)}}}{\sqrt{1+t}}\left(\frac{\eta(t)^2\log(2+t)}{(1+t)\sqrt{1+t}} + \frac{E_0}{\sqrt{1+t}} +  \int_0^t \frac{\eta(t)^2}{\sqrt{1+s} \sqrt{s} \sqrt{1+t-s}} ds\right)\\ &\leq C\left(E_0+\eta(t)^2\right)\frac{\log(2+t)}{1+t} e^{-\frac{|\beta \xx + \gamma y + \alpha \beta t|^2}{M(1 + t)}},\\
\left\|D_{\xx,y}^{\mathfrak{c}} v(\xx,y,t)\right\| &\leq C \frac{e^{-\frac{|\beta \xx + \gamma y + \alpha \beta t|^2}{M(1 + t)}}}{\sqrt{t}}\left(\frac{\eta(t)^2}{1+t} + \frac{E_0}{\sqrt{1+t}} + \eta(t)^2\! \! \int_0^t \!\!\frac{1}{\sqrt{s} \sqrt{t-s}} ds\right)\\& \leq C\left(E_0+\eta(t)^2\right)\frac{e^{-\frac{|\beta \xx + \gamma y + \alpha \beta t|^2}{M(1 + t)}}}{\sqrt{t}},\\
\left\|\psi(\xx,y,t)\right\| &\leq C \frac{e^{-\frac{|\beta \xx + \gamma y + \alpha \beta t|^2}{M(1 + t)}}}{\sqrt{1+t}}\left(E_0 + \eta(t)^2 \int_0^t \frac{\left(\log(2+s)\right)^2}{(1+s)\sqrt{1+s}} ds\right)\\ &\leq C\left(E_0 + \eta(t)^2\right)\frac{e^{-\frac{|\beta \xx + \gamma y + \alpha \beta t|^2}{M(1 + t)}}}{\sqrt{1+t}},\\
\left\|D_{\xx,y,t}^{\mathfrak{b}} \psi(\xx,y,t)\right\| &\leq C \frac{e^{-\frac{|\beta \xx + \gamma y + \alpha \beta t|^2}{M(1 + t)}}}{\sqrt{1+t}}\left(\frac{E_0}{\sqrt{1+t}} + \eta(t)^2 \int_0^t \frac{\left(\log(2+s)\right)^2}{(1+s)\sqrt{1+s}\sqrt{1 + t - s}} ds\right) \\ &\leq C\left(E_0 + \eta(t)^2\right)\frac{e^{-\frac{|\beta \xx + \gamma y + \alpha \beta t|^2}{M(1 + t)}}}{1+t},
\end{split} \label{iteration1}
\end{align}
with $\xx,y \in \R$, and multi-indices $\mathfrak{c} \in \Z_{\geq 0}^2$ and $\mathfrak{b} \in \Z_{\geq 0}^3$ satisfying $|\mathfrak{c}| = 1$ and $0 \leq |\mathfrak{b}| \leq 3$.

Next, we use~\eqref{conversion} to bound $\vt(\xx,y,t)$, $\vt_\xx(\xx,y,t)$ and $\vt_y(\xx,y,t)$ in terms of $\eta(t)$. 
Estimates~\eqref{conversion} and~\eqref{iteration1}, and $\eta(t)\leq B$ yield
\begin{align}
\begin{split}
\left\|\vt(\xx,y,t)\right\| \leq &\ \left\|v(\xx,y,t)\right\| + \left(\left\|u_\infty'\right\|_{\infty} + \left\|\vt_\xx(t)\right\|_{\infty} \right) \left\|\psi(\xx,y,t)\right\| \leq C\left(E_0 + \eta(t)^2\right)\frac{e^{-\frac{|\beta \xx + \gamma y + \alpha \beta t|^2}{M(1 + t)}}}{\sqrt{1+t}},\\
\left\|\vt_\xx(\xx,y,t)\right\| \leq &\ \left\|v_\xx(\xx,y,t)\right\| + \left(\left\|u_\infty''\right\|_{\infty} + \left\|\vt_{\xx \xx}(t)\right\|_{\infty}\right)\left\|\psi(\xx,y,t)\right\|\\ 
&\qquad + \left(\left\|u_\infty'\right\|_{\infty} + \left\|\vt_\xx(t)\right\|_{\infty} \right) \left\|\psi_\xx(\xx,y,t)\right\| \leq C\left(E_0+\eta(t)^2\right)\frac{e^{-\frac{|\beta \xx + \gamma y + \alpha \beta t|^2}{M(1 + t)}}}{\sqrt{t}},\\
\left\|\vt_y(\xx,y,t)\right\| \leq &\ \left\|v_y(\xx,y,t)\right\| + \left\|\vt_{\xx y}(t)\right\|_{\infty}\left\|\psi(\xx,y,t)\right\| +  \left(\left\|u_\infty'\right\|_{\infty} + \left\|\vt_\xx(t)\right\|_{\infty} \right) \left\|\psi_y(\xx,y,t)\right\| \\ \leq &\ C\left(E_0+\eta(t)^2\right)\frac{e^{-\frac{|\beta \xx + \gamma y + \alpha \beta t|^2}{M(1 + t)}}}{\sqrt{t}},
\end{split} \label{iteration3}
\end{align}
with $\xx,y \in \R$.

Subsequently, we bound solutions to the integral system~\eqref{integralscheme3}. First, we establish a short-time bound on $\vt_\xx$. As $\eta(t) \leq B$, estimates~\eqref{nonlest3} and~\eqref{iteration3} yield
\begin{align}
\begin{split}
\left\|\NT_1(\xt,\yt,s)\right\| &\leq C\left(E_0 + \eta(s)^2\right) \frac{e^{-\frac{|\beta \xt + \gamma \yt + \alpha \beta s|^2}{M(1+s)}}}{\sqrt{1+s}},
\end{split} \label{nonlinest21}
\end{align}
with $\xt,\yt \in \R$ and $0 \leq s \leq t$. On the other hand, the pointwise Green's function estimates in Theorem~\ref{theope} imply
\begin{align}
\begin{split}
\left\|G(\xx,\xt,y,s)\right\| &\leq C\frac{e^{-\frac{|\xx-\xt + \alpha s|^2 + |y|^2}{Ms}}}{s},
\end{split} \qquad \qquad \xx,\xt,y \in \R, s > 0. \label{pointwiseestimates21}
\end{align}
Therefore, applying~\eqref{nonlinest21} and~\eqref{pointwiseestimates21} to the integral equation for $\vt_\xx$ in~\eqref{integralscheme3} and employing~\eqref{intid4}, we obtain
\begin{align}
\begin{split}
\left\|\vt_\xx(\xx,y,t)\right\| &\leq C e^{-\frac{|\beta \xx + \gamma y + \alpha \beta t|^2}{M(1 + t)}} \left(\frac{E_0}{\sqrt{1+t}} + \left(E_0 + \eta(t)^2\right) \int_0^t \frac{1}{\sqrt{1+t}} ds\right)\\ & \leq C\left(E_0+\eta(t)^2\right)\left(1+\sqrt{t}\right)e^{-\frac{|\beta \xx + \gamma y + \alpha \beta t|^2}{M(1 + t)}},
\end{split} \label{iteration21}
\end{align}
with $\xx,y \in \R$. Thus, combining~\eqref{iteration21} (for short time) with~\eqref{iteration3} (for large time), we arrive at
\begin{align}
\begin{split}
\left\|\vt_\xx(\xx,y,t)\right\| & \leq C\left(E_0+\eta(t)^2\right)\frac{e^{-\frac{|\beta \xx + \gamma y + \alpha \beta t|^2}{M(1 + t)}}}{\sqrt{1+t}}, \qquad \xx,y\in\R.
\end{split} \label{iteration2}
\end{align}
Second, we bound the third derivatives $\vt_{\xx\xx\xx}$ and $\vt_{\xx \xx y}$. As $\eta(t) \leq B$, estimates~\eqref{nonlest33},~\eqref{iteration3} and~\eqref{iteration2} imply
\begin{align}
\begin{split}
\left\|\NT_2(\xt,\yt,s)\right\| \leq C\left(E_0 + \eta(s)^2\right)\left(1+\frac{1}{\sqrt{s}}\right) \frac{e^{-\frac{|\beta \xt + \gamma \yt + \alpha \beta s|^2}{M(1+s)}}}{\sqrt{1+s}},
\end{split} \label{nonlinest22}
\end{align}
with $\xt,\yt \in \R$ and $0 < s \leq t$. Moreover, Theorem~\ref{theope} yields
\begin{align}
\begin{split}
\left\|\partial_\xx G(\xx,\xt,y,s)\right\| &\leq C\left(1+\frac{1}{\sqrt{s}}\right)\frac{e^{-\frac{|\xx-\xt + \alpha s|^2 + |y|^2}{Ms}}}{s},\\
\left\|\partial_y G(\xx,\xt,y,s)\right\| &\leq C\frac{e^{-\frac{|\xx-\xt + \alpha s|^2 + |y|^2}{Ms}}}{s\sqrt{s}},
\end{split} \qquad \qquad \xx,\xt,y \in \R, s > 0. \label{pointwiseestimates22}
\end{align}
Therefore, applying~\eqref{nonlinest22} and~\eqref{pointwiseestimates22} to the integral system~\eqref{integralscheme3} and employing~\eqref{intid4}, we obtain
\begin{align}
\begin{split}
\left\|D_{\xx,y}^{\mathfrak{a}} \vt_{\xx \xx}(\xx,y,t)\right\| &\!\leq\! C \frac{e^{-\frac{|\beta \xx + \gamma y + \alpha \beta t|^2}{M(1 + t)}}}{\sqrt{1+t}}\left(E_0\left(1\!+\!\frac{1}{\sqrt{t}}\right) + \left(E_0 + \eta(t)^2\right)\! \int_0^t \!\left(1\! +\! \frac{1}{\sqrt{s}}\right) \left(1\!+\!\frac{1}{\sqrt{t-s}}\right)ds\right) \\ &\!\leq\! C\left(\sqrt{t} + \frac{1}{\sqrt{t}}\right)\left(E_0+\eta(t)^2\right),
\end{split} \label{iteration22}
\end{align}
with $\xx,y \in \R$ and $\mathfrak{a} \in \Z_{\geq 0}^2$ satisfying $|\mathfrak{a}| = 1$. Combining the interpolation inequality
\begin{align*} \|u'\|_\infty^2 \leq C\|u\|_\infty \|u''\|_\infty, \qquad u \in W^{2,\infty}(\R,\R),\end{align*}
with the estimates on $\vt_\xx, \vt_{\xx\xx\xx}, \vt_y$ and $\vt_{\xx\xx y}$, established in~\eqref{iteration3} and~\eqref{iteration22}, we arrive at
\begin{align} \|\vt_{\xx\xx}(t)\|_\infty, \|\vt_{\xx y}(t)\|_\infty \leq C\left(E_0+\eta(t)^2\right)\frac{1 + \sqrt{t}}{\sqrt{t}}.\label{iteration31}\end{align}

Now, combining estimates~\eqref{conversion},~\eqref{iteration1},~\eqref{iteration2},~\eqref{iteration22} and~\eqref{iteration31} yields
\begin{align}
\begin{split}
\left\|v_{\xx\xx}(t)\right\|_\infty \leq &\ \left\|\vt_{\xx\xx}(t)\right\|_\infty + C\sum_{j = 0}^2 \left(\left\|\partial_\xx^{j+1} u_\infty(t)\right\|_\infty + \left\|\partial_\xx^{j+1} \vt(t)\right\|_{\infty}\right)\left\|\partial_\xx^{2-j} \psi(t)\right\|_\infty\\
\leq &\ C\left(E_0+\eta(t)^2\right)\frac{1+\sqrt{t}}{\sqrt{t}},\\
\left\|v_{\xx y}(t)\right\|_\infty \leq &\ \left\|\vt_{\xx y}(t)\right\|_\infty + C\left(\left\|\vt_{\xx \xx y}(t)\right\|_{\infty}\left\|\psi(t)\right\|_\infty + \left\|\vt_{\xx y}(t)\right\|_{\infty} \left\|\psi_\xx(t)\right\|_\infty\right.\\
&\qquad \left.+ \left(\left\|\partial_{\xx\xx} u_\infty\right\|_{\infty} + \left\|\vt_{\xx\xx}(t)\right\|_{\infty} \right) \left\|\psi_y(t)\right\|_\infty + \left(\left\|\partial_\xx u_\infty\right\|_{\infty} + \left\|\vt_\xx(t)\right\|_{\infty} \right) \left\|\psi_{\xx y}(t)\right\|_\infty\right)
\\ \leq &\ C\left(E_0+\eta(t)^2\right)\frac{1+\sqrt{t}}{\sqrt{t}}.
\end{split} \label{iteration33}
\end{align}
Finally, estimate~\eqref{etaest} follows from~\eqref{iteration1},~\eqref{iteration3},~\eqref{iteration2},~\eqref{iteration31} and~\eqref{iteration33}, which concludes the proof.
\end{proof}

The proof of Theorem~\ref{maintheorem2} concerning exponentially localized perturbations has a similar set-up as the proof of Theorem~\ref{maintheorem}. However, the obtained decay estimates differ.

\begin{proof}[Proof of Theorem~\ref{maintheorem2}] In this proof, the constant $M > 1$ is as in Theorem~\ref{theope}, whereas $C>1$ denotes a constant, which is independent of $E_0,\xx,\xt,y$ and $t$, that will be taken larger if necessary.

Similar to the proof Theorem~\ref{maintheorem}, we conclude that the template function
\begin{align*}
\eta(t) :=& \sup_{\begin{smallmatrix} 0 \leq s \leq t\\ \xx,y \in \R\end{smallmatrix}} e^{\frac{|\xx + \alpha s|^2 + y^2}{M(1 + s)}} \left[(1+s)\sqrt{1+s}\left(\left\|v(\xx,y,s)\right\| + \sum_{1 \leq |\mathfrak{b}| \leq 3} \left\|D_{\xx,y,s}^\mathfrak{b} \psi(\xx,y,s)\right\|\right) \right.\\
&\quad \left.\phantom{\sum_{1 \leq |\mathfrak{b}| \leq 3}} + (1+s)\left(\left\|\vt(\xx,y,s)\right\| + \left\|\psi(\xx,y,s)\right\|\right) + \frac{\sqrt{s}\sqrt{1+s}}{\log(2+s)}\left(\left\|\vt_\xx(\xx,y,s)\right\| + \left\|v_\xx(\xx,y,s)\right\|\right)\right],
\end{align*}
is well-defined and continuous on some maximal interval $[0,T)$ with $T \in (0,\infty]$ and, if $T$ is finite, it must blow up as $t \uparrow T$. Again we prove that there exists constants $B > 0$ and $C > 1$ such that for all $t \geq 0$ with $\eta(t) \leq B$ we have
\begin{align} \eta(t) \leq C\left(E_0 + \eta(t)^2\right). \label{etaest2}\end{align}
Then, provided $E_0 < \frac{1}{4C}$, we have $\eta(t) \leq 2CE_0$ for \emph{all} $t \geq 0$ by continuous induction, which yields the result.

Thus, we establish the key estimate~\eqref{etaest2} proceeding as in the proof Theorem~\ref{maintheorem}. This time, we consider the closed integral scheme, which consists of the $v$- and $\psi$-equations in~\eqref{integralscheme2}, the $\vt_\xx$- and $\vt_{\xx\xx}$-equations in~\eqref{integralscheme3} and the estimates~\eqref{nonlest3},~\eqref{conversion},~\eqref{nonlest1} and~\eqref{nonlest2}. Moreover, we employ two simpler integral identities
\begin{align}
\begin{split}
\int_{\R^2} e^{-\frac{|\xx-\xt + \alpha t|^2 + |y-\yt|^2}{Mt} - \frac{\xt^2 + \yt^2}{M}}d\xt d\yt &= \frac{M\pi t e^{-\frac{|\xx + \alpha t|^2 + y^2}{M(1 + t)}}}{1+t},\\
\int_{\R^2} e^{-\frac{|\xx-\xt + \alpha (t-s)|^2 + |y-\yt|^2}{M(t-s)} - \frac{|\xt + \alpha s|^2 + \yt^2}{M(1+s)}}d\xt d\yt &= \frac{M\pi (t-s) (1+s) e^{-\frac{|\xx + \alpha t|^2 + y^2}{M(1 + t)}}}{1+t},
\end{split} \label{intid42}
\end{align}
with $\xx,y,s,t \in \R$ satisfying $0 \leq s < t$. These identities are obtained by factorizing the double integral into a product of an integral over $\xt$ and one over $\yt$, each of which can be evaluated using~\eqref{intid}.

Let $B > 1$ be as in~\S\ref{secitscheme} and assume $t > 0$ is such that $\eta(t) \leq B$. As in the proof of Theorem~\ref{maintheorem}, we bound the solutions to the integral systems~\eqref{integralscheme2} and~\eqref{integralscheme3} in terms of $\eta(t)$.  Thus, by applying Theorem~\ref{theope},~\eqref{nonlest2} and~\eqref{intid42} to the integral system~\eqref{integralscheme2}, we obtain
\begin{align}
\begin{split}
\left\|v(\xx,y,t)\right\| &\leq C \frac{e^{-\frac{|\xx + \alpha t|^2 + y^2}{M(1 + t)}}}{1+t}\left(\frac{\eta(t)^2}{(1+t)^2} + \frac{E_0}{\sqrt{1+t}} + \eta(t)^2 \int_0^t \frac{\log(2+s)}{(1+s) \sqrt{s} \sqrt{t-s}} ds\right)\\ &\leq C\left(E_0+\eta(t)^2\right)\frac{e^{-\frac{|\xx + \alpha t|^2 + y^2}{M(1 + t)}}}{(1+t)\sqrt{1+t}},\\
\left\|\psi(\xx,y,t)\right\| &\leq C \frac{e^{-\frac{|\xx + \alpha t|^2 + y^2}{M(1 + t)}}}{1+t}\left(E_0 + \eta(t)^2 \int_0^t \frac{1}{(1+s)^2} ds\right) \leq C\left(E_0 + \eta(t)^2\right)\frac{e^{-\frac{|\xx + \alpha t|^2 + y^2}{M(1 + t)}}}{1+t},\\
\left\|D_{\xx,y,t}^{\mathfrak{b}} \psi(\xx,y,t)\right\| &\leq C \frac{e^{-\frac{|\xx + \alpha t|^2 + y^2}{M(1 + t)}}}{1+t}\left(\frac{E_0}{\sqrt{1+t}} + \eta(t)^2 \int_0^t \frac{1}{(1+s)^2\sqrt{1 + t - s}} ds\right) \\ &\leq C\left(E_0 + \eta(t)^2\right)\frac{e^{-\frac{|\xx + \alpha t|^2 + y^2}{M(1 + t)}}}{(1+t)\sqrt{1+t}},
\end{split} \label{iteration12}
\end{align}
with $\xx,y \in \R$ and $\mathfrak{b} \in \Z_{\geq 0}^3$ with $1 \leq |\mathfrak{b}| \leq 3$. Similarly, by applying~\eqref{nonlest3},~\eqref{pointwiseestimates22} and~\eqref{intid42} to the integral system~\eqref{integralscheme3}, we bound
\begin{align}
\begin{split}
\left\|\vt_\xx(\xx,y,t)\right\| &\leq C \frac{e^{-\frac{|\xx + \alpha t|^2 + y^2}{M(1 + t)}}}{1+t} \left(E_0 \left(1+\frac{1}{\sqrt{t}}\right) + \eta(t)^2 \int_0^t \frac{1}{(1+s)} \left(1+\frac{1}{\sqrt{t-s}}\right) ds\right)\\ & \leq C\left(E_0+\eta(t)^2\right)e^{-\frac{|\xx + \alpha t|^2 + y^2}{M(1 + t)}}\frac{\log(2+t)}{\sqrt{1+t}\sqrt{t}},\\
\left\|\vt_y(\xx,y,t)\right\| &\leq C \frac{e^{-\frac{|\xx + \alpha t|^2 + y^2}{M(1 + t)}}}{1+t} \left(\frac{E_0}{\sqrt{t}} + \eta(t)^2 \int_0^t \frac{1}{(1+s)\sqrt{t-s}} ds\right)\\ & \leq C\left(E_0+\eta(t)^2\right)e^{-\frac{|\xx + \alpha t|^2 + y^2}{M(1 + t)}}\frac{\log(2+t)}{(1+t)\sqrt{t}},\\
\left\|\vt_{\xx\xx}(\xx,y,t)\right\| &\leq C \frac{e^{-\frac{|\xx + \alpha t|^2 + y^2}{M(1 + t)}}}{1+t} \left(E_0 \left(1+\frac{1}{\sqrt{t}}\right) + \int_0^t \left(1+\frac{\log(2+s)}{\sqrt{1+s}\sqrt{s}}\right) \left(1+\frac{1}{\sqrt{t-s}}\right) ds\right) \leq C,
\end{split} \label{iteration222}
\end{align}
with $\xx,y \in \R$. Finally, we use $\psi(s) = 0$ for $0 \leq s \leq 1$, $\eta(t)\leq B$,~\eqref{conversion},~\eqref{iteration12} and~\eqref{iteration222} to bound
\begin{align}
\begin{split}
\left\|\vt(\xx,y,t)\right\| \leq &\ \left\|v(\xx,y,t)\right\| + \left(\left\|u_\infty'\right\|_{\infty} + \left\|\vt_\xx(t)\right\|_{\infty} \right) \left\|\psi(\xx,y,t)\right\| \leq C\left(E_0 + \eta(t)^2\right)\frac{e^{-\frac{|\xx + \alpha t|^2 + y^2}{M(1 + t)}}}{1+t},\\
\left\|v_\xx(\xx,y,t)\right\| \leq &\ \left\|\vt_\xx(\xx,y,t)\right\| + \left(\left\|u_\infty''\right\|_{\infty} + \left\|\vt_{\xx \xx}(t)\right\|_{\infty}\right)\left\|\psi(\xx,y,t)\right\|\\ 
&\quad +  \left(\left\|u_\infty'\right\|_{\infty} + \left\|\vt_\xx(t)\right\|_{\infty} \right) \left\|\psi_\xx(\xx,y,t)\right\|
\leq C\left(E_0 + \eta(t)^2\right)e^{-\frac{|\xx + \alpha t|^2 + y^2}{M(1 + t)}}\frac{\log(2+t)}{\sqrt{t}\sqrt{1+t}},
\end{split} \label{iteration32}
\end{align}
with $\xx,y \in \R$. Estimate~\eqref{etaest2} follows from~\eqref{iteration12},~\eqref{iteration222} and~\eqref{iteration32}, which yields the result.
\end{proof}

\begin{rem} \label{optimality}
It follows from the decomposition of the Green's function $G(\xx,\xt,y,t)$ in Theorem~\ref{theope} that the leading-order linear contribution in the integral equation~\eqref{integralscheme4} for $\vt(\xx,y,t)$ comes from the term
\begin{align}
u_\infty(\xx) \int_{\R^2}  e(\xx,\xt,y-\yt,t) v_0(\xt,\yt) d\xt d\yt. \label{linearest}
\end{align}
Using the explicit formula for the Green's function $e(\xx,\xt,y,t)$ given in Theorem~\ref{theope}, one readily observes from~\eqref{intid4} and~\eqref{intid42} that for $v_0 \in W^{2,\infty}(\R^2,\R^n)$ satisfying~\eqref{v01} the optimal decay rate of~\eqref{linearest} is $t^{-1/2}$, whereas for $v_0 \in W^{1,\infty}(\R^2,\R^n)$ satisfying~\eqref{v02} the optimal decay rate of~\eqref{linearest} is $t^{-1}$. Thus, the decay rates of $\vt(\xx,y,t) = \ut(\xx,y,t) - u_\infty(\xx)$ in Theorems~\ref{maintheorem} and~\ref{maintheorem2} are optimal. Similarly, one can verify the optimality of the other decay rates in Theorems~\ref{maintheorem} and~\ref{maintheorem2}. One finds that we have obtained a (possibly) suboptimal decay rate \emph{only} in estimating $v(\xx,y,t) = \ut(\xx + \psi(\xx,y,t),y,t) - u_\infty(\xx)$ in Theorem~\ref{maintheorem}, in estimating $\vt_\xx(\xx,y,t) = \ut_\xx(\xx,y,t) - u_\infty'(\xx)$ in Theorem~\ref{maintheorem2} and in estimating $\vt_y(\xx,y,t) = \ut_y(\xx,y,t)$ in Theorems~\ref{maintheorem} and~\ref{maintheorem2}. For the first, we establish a decay rate of $\log(2+t)t^{-1}$ in Theorem~\ref{maintheorem}, whereas one would expect that the optimal decay rate is $t^{-1}$ from the estimation of the leading-order linear contribution. For the second, we establish a decay rate of $\log(2+t)t^{-1}$ in Theorem~\ref{maintheorem2}, whereas one would expect $t^{-1}$. Finally, for the third,  we establish decay rates of $t^{-1/2}$ in Theorem~\ref{maintheorem} and of $\log(2+t)t^{-3/2}$ in Theorem~\ref{maintheorem2}, whereas one would expect that the optimal decay rates are $t^{-1}$ and $t^{-3/2}$, respectively. This difference of a factor $t^{-1/2}$, as well as the (possibly) artificial logarithm, might come from our approach to dealing with derivatives of the perturbation in the nonlinearity -- see Remark~\ref{motvt}.
\end{rem}

\appendix
\renewcommand*{\thesection}{\Alph{section}}

\section{Proof of pointwise estimates} \label{proofpointwise}

\subsection{Approach} \label{sec:papproach}

We prove Theorem~\ref{theope} by employing similar methods as in~\cite{JUN}, where pointwise Green's function bounds are obtained for wave trains in reaction-diffusion systems in one spatial dimension. In the proof, we distinguish between two different regimes, that is we obtain pointwise Green's function estimates for $(|\xx - \xt + \alpha t|+|y|)/t$ large and for bounded $(|\xx - \xt + \alpha t|+|y|)/t$, separately.

In the first regime, we proceed as in~\cite[Proposition 2.7]{HOF}. We apply the Laplace transform to~\eqref{Fourier} and write the solution operator $e^{\El_{\nu_y} t}$ as an integral
\begin{align}
 e^{\El_{\nu_y} t} = \int_{\Gamma_\nu} e^{\lambda t} \left(\El_{\nu_y} - \lambda\right)^{-1} d\lambda, \label{LaplaceSol}
\end{align}
for $\nu_y \in \C$, where $\Gamma_{\nu_y}$ is a contour in the resolvent set $\rho(\El_{\nu_y})$ containing the spectrum $\sigma(\El_{\nu_y})$. Consequently, the temporal Green's function $G(\xx,\xt,y,t)$ can be related via the Fourier and Laplace transforms to the resolvent kernel $G_{\nu_y,\lambda}(\xx,\xt)$ associated with the elliptic operator $\El_{\nu_y} - \lambda$. Thus, using~\eqref{interm} and~\eqref{LaplaceSol} we obtain
\begin{align} G(\xx,\xt,y,t) = \frac{1}{2\pi i \sqrt{2\pi}} \int_{\R} \int_{\Gamma_{\nu_y}} e^{i\nu_y y + \lambda t} G_{\nu_y,\lambda}(\xx,\xt)d\lambda d\nu_y. \label{Duhamel}\end{align}
We regard the eigenvalue problem $\El_{\nu_y}u = \lambda u$, associated with the resolvent kernel $G_{\nu_y,\lambda}(\xx,\xt)$, as an ODE in the \emph{complex} parameters $\nu_y$ and $\lambda$. The behavior of this ODE is well-understood when $\nu_y$ and $\lambda$ are contained in a certain set $\Es \subset \C^2$, sufficiently far away from the origin. Thus, for $(\nu_y,\lambda) \in \Es$ pointwise estimates on the resolvent kernel $G_{\nu_y,\lambda}(\xx,\xt)$ can be obtained. These estimates can then be employed to bound the temporal Green's function $G(\xx,\xt,y,t)$ by deforming the contour integrals in $\lambda$ and $\nu_y$ in~\eqref{Duhamel} to lie in $\Es$. We emphasize that, in order to obtain this bound on $G(\xx,\xt,y,t)$, it is crucial that $(|\xx - \xt + \alpha t|+|y|)/t$ is sufficiently large to compensate for the fact that the resulting $\lambda$-contour lies partly in the right half-plane, far away from the origin, causing the factor $e^{\lambda t}$ in~\eqref{Duhamel} to blow-up as $t \to \infty$.

In order to obtain sharp pointwise Green's function estimates for $(|\xx - \xt + \alpha t|+|y|)/t$ bounded, it is necessary to minimize the factor $e^{\lambda t}$ in~\eqref{Duhamel}. This can be established by choosing the contour $\Gamma_{\nu_y} \subset \rho(\El_{\nu_y})$ to lie as much into the left half plane as possible. However, since the spectrum of $\El_0$ touches the imaginary axis at the origin, the obtained decay is not sufficient to close an eventual nonlinear iteration argument -- see Remark~\ref{remtranslational}. Instead, we proceed as in~\cite{JONZ,JUN,OHZUM} and split off the slow-decaying translational mode first using the Fourier-Bloch decomposition~\eqref{Bloch} of the solution operator. More precisely, we split~\eqref{Bloch} into an integral over a small neighborhood $U_0 \subset \Omega$ of $0$, corresponding to the translational part, and a residual integral over $\Omega \setminus U_0$ to which we apply the Laplace transform. Since the resolvent set $\rho(L_{\pmb \nu})$ is contained in $\Re(\lambda) \leq -\eta$ for ${\pmb \nu} \in \Omega \setminus U_0$ by (D3), we can deform the contour to obtain exponential decay from the factor $e^{\lambda t}$ in the Laplace transform formula. Finally, we approximate the translational part by $u_\infty(\xx)e(\xx,\xt,y,t)$ using the expansion in Lemma~\ref{Lem1}.

This section is structured as follows. First, we establish pointwise estimates on the resolvent kernels associated with the operators $\El_{\nu_y} - \lambda,\nu_y \in \R$ and $L_{\pmb \nu} - \lambda,{\pmb \nu} \in \Omega$. Then, we recall some standard, but nontrivial, integrals that we will encounter in our estimation of the Green's function. Finally, we obtain pointwise estimates on the temporal Green's function for $(|\xx - \xt + \alpha t|+|y|)/t$ large and $(|\xx - \xt + \alpha t|+|y|)/t$ bounded, separately.

\subsection{Bounds on the resolvent kernel}

We obtain bounds on the Green's function $G_{\nu_y,\lambda}(\xx,\xt)$ associated with the elliptic operator $\El_{\nu_y} - \lambda$ for $(\lambda,\nu_y)$ in a certain subset of $\C^2$, sufficiently far away from the origin. We proceed as in~\cite[Section 3.2]{BEC} and obtain these bounds from the behavior of the spatial eigenvalues associated with the temporal eigenvalue problem $\El_{\nu_y} u =  \lambda u$.

\begin{lem} \label{high}
Complexify $\nu_y$ and consider the family $\El_{\nu_y}, \nu_y \in \C$ of operators on $L^2(\R,\C^n)$ given by~\eqref{Fourier}. Denote by $G_{\nu_y,\lambda}(\xx,\xt)$ the Green's function associated with $\El_{\nu_y} - \lambda$. There exists $\kappa_1 \in (0,1)$ and $C,R,\kappa_{2} > 1$, such that for all $(\nu_y,\lambda) \in \C^2$ satisfying
\begin{align}
|\nu_y| + |\lambda|^{1/2} > R, \quad \Re(\lambda) \geq - \kappa_1\left(|\Re(\nu_y)|^2 + |\Im(\lambda)|\right) + \kappa_2 |\Im(\nu_y)|^2, \label{sector2}
\end{align}
$\lambda$ is in the resolvent set $\rho(L_{\nu_y})$ and the Green's function may be bounded as follows
\begin{align}
\begin{split}
\left\|D_{\xx,\xt}^\mathfrak{a} G_{\nu_y,\lambda}(\xx,\xt)\right\| &\leq C\frac{e^{-\left(|\lambda|^{1/2} + |\nu_y|\right)|\xx-\xt|/C}}{\left(|\lambda|^{1/2} + |\nu_y|\right)^{1-|\mathfrak{a}|}},
\end{split} \qquad \xx,\xt \in \R, \label{high2}\end{align}
where $\mathfrak{a} \in \Z_{\geq 0}^2$ is a multi-index with $0 \leq |\mathfrak{a}| \leq 1$.
\end{lem}
\begin{proof} Throughout this proof, we denote by $C>1$ a constant, which is independent of $\xx,\xt,\lambda$ and $\nu_y$, that will be taken larger if necessary.

Let $f \in L^2(\R,\C^{n})$. We rescale the inhomogeneous problem $(\El_{\nu_y} - \lambda) u = f$ by setting $z = (|\lambda|^{1/2} + |\nu_y|)\xx$ and $\phi = (u,\partial_ z u)$. This leads to the equivalent problem
\begin{align}
\left(\partial_z - \A(z,\nu_y,\lambda)\right)\phi = g(z,\nu_y,\lambda), \quad g(z,\nu_y,\lambda) := \left(\begin{array}{c} 0 \\ \frac{D^{-1} f\left(\left(|\lambda|^{1/2} + |\nu_y|\right)^{-1} z\right)}{\left(|\lambda|^{1/2} + |\nu_y|\right)^2} \end{array}\right), \label{ODE3}
\end{align}
where
\begin{align*}
\A(z,\nu_y,\lambda) := A(\nu_y,\lambda) + \frac{1}{|\lambda|^{1/2} + |\nu_y|}B(z,\nu_y,\lambda), \quad A(\nu_y,\lambda) := \left(\begin{array}{cc} 0 & \I_n\\ \frac{k^{-2}(\nu_y^2 + D^{-1}\lambda)}{(|\lambda|^{1/2} + |\nu_y|)^2} & 0 \end{array}\right),
\end{align*}
and $B$ is bounded on $\R \times \left\{(\nu_y,\lambda) \in \C^2 : |\nu_y| + |\lambda|^{1/2} > R\right\}$. Since $D^{-1}$ is symmetric and positive-definite, there exists $d_{\mathrm{min}} > 0$ such that $v^* D^{-1} v \geq d_{\mathrm{min}}$ for all $v \in \C^n$ with $\|v\| = 1$. We use the latter to explicitly bound the eigenvalues of the matrix $A(\nu_y,\lambda)$, which are given by the square roots of the eigenvalues of its lower-left block. Thus, take $\kappa_1 \in (0,1)$ and $\kappa_2 > 0$ such that $4\kappa_1\|D^{-1}\| < \min\{2d_{\mathrm{min}},1\}$ and $2\kappa_2 > 5\|D^{-1}\|$. We find that for any $(\nu_y,\lambda) \in \C^2$ satisfying~\eqref{sector2} it holds
\begin{align*}
\left|\Re(\sigma(A(\nu_y,\lambda)))\right| \geq \frac{1}{2k} \sqrt{\min\left\{\tfrac{1}{4}-\left\|D^{-1}\right\|\kappa_1,\tfrac{d_{\mathrm{min}}}{2}\right\}},
\end{align*}
i.e.~$A(\nu_y,\lambda)$ is hyperbolic with $(\nu_y,\lambda)$-independent spectral gap. We apply~\cite[Lemma 1.2]{SAN1993} (see also~\cite{COP}) to the homogeneous problem,
\begin{align}
\left(\partial_z - \A(z,\nu_y,\lambda)\right)\phi = 0. \label{ODE2}
\end{align}
Thus, taking $R > 1$ sufficiently large, system~\eqref{ODE2} has for all $(\nu_y,\lambda) \in \C^2$ satisfying~\eqref{sector2} an exponential dichotomy on $\R$ with constants independent of $\lambda$ and $\nu_y$. Denote by $T^{s,u}_{\nu_y,\lambda}(z,\bar{z})$ the stable and unstable evolution of~\eqref{ODE2} under the exponential dichotomy. Variation of constants yields that the unique solution $\phi \in L^2(\R,\C^n)$ to~\eqref{ODE3} is given by
\begin{align*}
\phi(z,\nu_y,\lambda) = \int_{-\infty}^z T^s_{\nu_y,\lambda}(z,\bar{z}) g(\bar{z},\nu_y,\lambda)d\bar{z} - \int_z^\infty T^u_{\nu_y,\lambda}(z,\bar{z}) g(\bar{z},\nu_y,\lambda)d\bar{z}.
\end{align*}
Thus, rescaling back to the original coordinates, we obtain
\begin{align*}
\left(\begin{array}{c} G_{\nu_y,\lambda}(\xx,\xt)\\
      \partial_\xx G_{\nu_y,\lambda}(\xx,\xt) \end{array}\right)
= \left(\begin{array}{cc} \frac{\I_n}{|\lambda|^{1/2} + |\nu_y|} & 0 \\ 0 & \I_n \end{array}\right) \cdot \begin{cases} T^s_{\nu_y,\lambda}\left(\left(|\lambda|^{1/2} + |\nu_y|\right)\xx,\left(|\lambda|^{1/2} + |\nu_y|\right)\xt\right) \left(\begin{array}{c} 0 \\ D^{-1}\end{array}\right), & \xx \geq \xt,\\
-T^u_{\nu_y,\lambda}\left(\left(|\lambda|^{1/2} + |\nu_y|\right)\xx,\left(|\lambda|^{1/2} + |\nu_y|\right)\xt\right) \left(\begin{array}{c} 0 \\ D^{-1}\end{array}\right), & \xx < \xt,\end{cases}
\end{align*}
which yields the estimates on $G_{\nu_y,\lambda}(\xx,\xt)$ and $\partial_\xx G_{\nu_y,\lambda}(\xx,\xt)$ in~\eqref{high2}. Finally, for the estimate on $\partial_{\xt} G_{\nu_y,\lambda}(\xx,\xt)$ we observe that $G_{\nu_y,\lambda}(\xx,\xt) = H_{\nu_y,\lambda}(\xt,\xx)^*$, where $H_{\nu_y,\lambda}(\xx,\xt)$ is the Green's function associated with the adjoint operator $(\El_{\nu_y} - \lambda)^*$ -- see also~\cite[Lemma 4.3]{ZUH}. Analogous to the above, one obtains
\begin{align*}\left\|\partial_\xx H_{\nu_y,\lambda}(\xx,\xt)\right\| \leq Ce^{-\left(|\lambda|^{1/2} + |\nu_y|\right)|\xx-\xt|/C}, \qquad \xx,\xt \in \R,
\end{align*}
for any $(\nu_y,\lambda) \in \C^2$ satisfying~\eqref{sector2} (by choosing $R > 1$ larger if necessary). This yields the bound on $\partial_\xt G_{\nu_y,\lambda}(\xx,\xt) = \left(\partial_\xt H_{\nu_y,\lambda}(\xt,\xx)\right)^*$ in~\eqref{high2}. \end{proof}

Using the Bloch transform, Lemma~\ref{high} yields pointwise bounds on the resolvent kernel $G_{{\pmb \nu},\lambda}(\xx,\xt)$ associated with the operator $L_{\pmb \nu} - \lambda$ given by~\eqref{Bloch2}.

\begin{cor} \label{high3}
There exists constants $C,R > 1$ and $\mu > 0$ such that for all $({\pmb \nu},\lambda) \in \Omega \times \C$ satisfying
\begin{align*} |\lambda| > R, \quad \Re(\lambda) \geq -\mu \left(\nu_y^2 + |\Im(\lambda)|\right),\end{align*}
$\lambda$ is in the resolvent set $\rho(L_{\pmb \nu})$ and the Green's function $\G_{{\pmb \nu},\lambda}(\xx,\xt)$ associated with the operator $L_{\pmb \nu} - \lambda$ satisfies
\begin{align*}
\begin{split}
\sup_{\nu_x \in [-\pi,\pi]} \left\|D_{\xx,\xt}^\mathfrak{a}\G_{{\pmb \nu},\lambda}(\xx,\xt)\right\| &\leq C|\lambda|^{\frac{|\mathfrak{a}|-1}{2}},
\end{split}
\qquad \xx,\xt \in [0,1],\end{align*}
where $\mathfrak{a} \in \Z_{\geq 0}^2$ is a multi-index with $0 \leq |\mathfrak{a}| \leq 1$.
\end{cor}
\begin{proof}
It holds $\G_{{\pmb \nu},\lambda}(\xx,\xt) = e^{i\nu_x(\xt-\xx)}G_{\nu_y,\lambda}(\xx,\xt)$ for any $({\pmb \nu},\lambda) \in \Omega \times \C$ and $\xx,\xt \in [0,1]$. The result now follows from Lemma~\ref{high}.
\end{proof}
\begin{rem} \label{mid}
Let $K$ be some compact set in $\Omega \times \C$ such that for pairs $({\pmb \nu},\lambda) \in K$ it holds $\lambda \in \rho(L_{\pmb \nu})$. By analyticity of $\G_{{\pmb \nu},\lambda}(\xx,\xt)$ and its $\xx$- and $\xt$-derivatives in $\lambda$ and ${\pmb \nu}$ there exists $C>1$ such that
\begin{align*} \sup_{({\pmb \nu},\lambda) \in K, \xx,\xt \in [0,1]} \left\|D_{\xx,\xt}^\mathfrak{a} \G_{{\pmb \nu},\lambda}(\xx,\xt)\right\| \leq C,\end{align*}
where $\mathfrak{a} \in \Z_{\geq 0}^2$ is a multi-index with $0 \leq |\mathfrak{a}| \leq 1$.
\end{rem}

\subsection{Some useful integrals}
In the pointwise estimation of the Green's function, we need to evaluate some standard, but nontrivial, integrals. First, recall
\begin{align} \int_{\R} e^{az(z-2b)}dz &= \frac{\sqrt{\pi}}{\sqrt{|a|}}e^{-ab^2}, &\quad a < 0, b \in \C, \label{intid1} \end{align}
from~\S\ref{sec:nonstab}. In addition, we use the identities
\begin{align}
\int_{\R} |z|^a e^{-b z^2} dz &= \Gamma\left(\frac{a+1}{2}\right)b^{-(a+1)/2}, &\quad a \geq 0, b > 0, \label{intid2}\\
\int_0^\infty \frac{e^{-bz}}{\sqrt{z}}dz &= \frac{\sqrt{\pi}}{\sqrt{b}}, \ \ \ &b > 0.\label{intid3}
\end{align}

\subsection{Pointwise Green's function estimates for large \texorpdfstring{$(|\xx - \xt + \alpha t|+|y|)/t$}{ }.}

As described in~\S\ref{sec:papproach}, we obtain pointwise bounds on the Green's function and its derivatives for large $(|\xx - \xt + \alpha t|+|y|)/t$ by proceeding as in~\cite[Proposition 2.7]{HOF}.

\begin{lem} \label{Point1}
Assume (D1)-(D3) hold true. There exists constants $m > 0$ and $S,C,M>1$ such that for any $\xx,\xt,y \in \R, t > 0$ satisfying
\begin{align*}
|\xx-\xt+\alpha t|+|y| \geq St,
\end{align*}
the Green's function $G(\xx,\xt,y,t)$ associated with the operator $\partial_t - \El$ enjoys the estimate
\begin{align*}
\begin{split}
\left\|D_{\xx,\xt,y}^\mathfrak{b} G(\xx,\xt,y,t)\right\| &\leq Ct^{-1-\frac{|\mathfrak{b}|}{2}} e^{-mt} e^{-\frac{|\xx-\xt + \alpha t|^2 + |y|^2}{Mt}},
\end{split} \qquad
\end{align*}
where $\mathfrak{b} \in \Z_{\geq 0}^3$ is a multi-index with $0 \leq |\mathfrak{b}| \leq 1$.
\end{lem}
\begin{proof} Throughout this proof, we denote by $C>1$ a constant, which is independent of $\xx,\xt,y$ and $t$, that will be taken larger if necessary.

Complexify $\nu_y$ and consider the family $\El_{\nu_y}, \nu_y \in \C$ of operators analytic in $\nu_y$ given by~\eqref{Fourier}. Let $R,\kappa_{1,2} > 0$ be as in Lemma~\ref{high} and take
\begin{align*} A := \epsilon_0^2 \frac{(|\xx - \xt + \alpha t| + |y|)^2}{t^2}, \quad a := \mathrm{sgn}(y)\sqrt{\frac{A}{\kappa_2}},\end{align*}
where $\epsilon_0 > 0$ is to be defined. We consider~\eqref{Duhamel} as a complex contour integral in $\nu_y$. By Cauchy's integral theorem we have
\begin{align}
\begin{split}
D_{\xx,\xt}^\mathfrak{a} G(\xx,\xt,y,t) &= \frac{1}{2\pi i \sqrt{2\pi}} \int_{\R} \int_{\Gamma_{\nu_y}} e^{i(z+ia)y + \lambda t} D_{\xx,\xt}^\mathfrak{a} G_{z+ia,\lambda}(\xx,\xt) d\lambda dz,\\
\partial_y G(\xx,\xt,y,t) &= \frac{1}{2\pi \sqrt{2\pi}} \int_{\R} \int_{\Gamma_{\nu_y}} (z+ia) e^{i(z+ia)y + \lambda t} G_{z+ia,\lambda}(\xx,\xt) d\lambda dz,
\end{split}
\label{Duhamel2}\end{align}
where $\mathfrak{a} \in \Z_{\geq 0}^2$ is a multi-index with $0 \leq |\mathfrak{a}| \leq 1$. We take the contour
\begin{align*} \Gamma_{\nu_y} := \{\lambda \in \C \colon \Re(\lambda) = A - \kappa_1(|\Re(\nu_y)|^2 + |\Im(\lambda)|)\}.\end{align*}
Note that for every $\lambda \in \Gamma_{\nu_y}$ it holds
\begin{align}|\nu_y|^2 + |\lambda| \geq \max\left\{|\Re(\lambda)|, \frac{A - \Re(\lambda)}{\kappa_1}\right\} \geq \frac{A}{1 + \kappa_1}. \label{normb}\end{align}
By~\eqref{normb} and the fact that $A = \kappa_2 a^2$, every $\nu_y \in \{z + ia : z \in \R\}$ and $\lambda \in \Gamma_{\nu_y}$ satisfy~\eqref{sector2} for $A > 0$ sufficiently large. Therefore, every $\lambda \in \Gamma_{\nu_y}$ is in the resolvent set of $\El_{\nu_y}$ for $\Im(\nu_y) = a$. Taking the lower bound $S > 0$ sufficiently large, we may assume
\begin{align}
\frac{|y| + |\xx-\xt|}{t} \geq \frac{|y| + |\xx - \xt + \alpha t|}{t} - |\alpha| \geq \frac{|y| + |\xx - \xt + \alpha t|}{2t} . \label{large}
\end{align}
Finally, by taking $\epsilon_0 > 0$ sufficiently small, we are able to estimate~\eqref{Duhamel2} using Lemma~\ref{high} and identities~\eqref{intid2},~\eqref{intid3},~\eqref{normb} and~\eqref{large}
\begin{align*} \|D^\mathfrak{a}_{\xx,\xt} G(\xx,\xt,y,t)\| &\leq C \int_{\R} \int_{\Gamma_{\nu_y}} e^{-a y + \Re(\lambda) t}\|D^\mathfrak{a}_{\xx,\xt} G_{z+ia,\lambda}(\xx,\xt)\| d\lambda dz\\
&\leq C \int_{\R} \int_{\R} e^{-a y + At - \sqrt{\frac{A}{1+\kappa_1}}|\xx-\xt|/C - \kappa_1 (|\Im(\lambda)| + z^2)t} |\Im(\lambda)|^{\frac{|\mathfrak{a}|-1}{2}} d|\Im(\lambda)| dz\\
&\leq C \int_{\R} \int_{\R} e^{(\epsilon_0^2 - \epsilon_0/C)\frac{(|y| + |\xx - \xt + \alpha t|)^2}{t} - \kappa_1 (|\Im(\lambda)| + z^2)t} |\Im(\lambda)|^{\frac{|\mathfrak{a}|-1}{2}} d|\Im(\lambda)| dz\\
&\leq Ct^{-1-\frac{|\mathfrak{a}|}{2}} e^{-mt} e^{-\frac{|\xx-\xt + \alpha t|^2 + |y|^2}{Mt}},
 \end{align*}
for some $m > 0$ and $M > 1$, where in the latter inequality we used that the lower bound $S > 0$ can be taken larger if necessary. Analogously, we estimate
\begin{align*} \|\partial_y G(\xx,\xt,y,t)\| &\leq C \int_{\R} \int_{\Gamma_{\nu_y}} \left(|z|+|a|\right)e^{-a y + \Re(\lambda) t}\| G_{z+ia,\lambda}(\xx,\xt)\| d\lambda dz\\
&\leq C \int_{\R} \int_{\R} \left(|z|+\sqrt{A}\right)e^{-a y + At - \sqrt{\frac{A}{1+\kappa_1}}|\xx-\xt|/C - \kappa_1 (|\Im(\lambda)| + z^2)t} |\Im(\lambda)|^{-\frac{1}{2}} d|\Im(\lambda)| dz\\
&\leq C \int_{\R} \int_{\R} \left(|z|+\frac{|y| + |\xx - \xt + \alpha t|}{t}\right)\frac{e^{(\epsilon_0^2 - \epsilon_0/C)\frac{(|y| + |\xx - \xt + \alpha t|)^2}{t} - \kappa_1 (|\Im(\lambda)| + z^2)t}}{\sqrt{|\Im(\lambda)|}} d|\Im(\lambda)| dz\\
&\leq Ct^{-\frac{3}{2}} e^{-m t} e^{-\frac{|\xx-\xt + \alpha t|^2 + |y|^2}{Mt}},
\end{align*}
for some $m > 0$ and $M > 1$.
\end{proof}

\subsection{Pointwise Green's function estimates for bounded \texorpdfstring{$(|\xx-\xt+\alpha t|+|y|)/t$}{ }}
As described in~\S\ref{sec:papproach}, we establish pointwise bounds on the temporal Green's function and its derivatives for bounded $(|\xx - \xt + \alpha t|+|y|)/t$ following~\cite{JUN,OHZUM}.
\begin{lem} \label{Point2}
Assume (D1)-(D3) hold true and let $S > 1$. There exists constants $C,M>1$ such that for any $\xx,\xt,y \in \R, t > 0$ satisfying
\begin{align*}
|\xx-\xt+\alpha t|+|y| \leq St,
\end{align*}
the Green's function $G(\xx,\xt,y,t)$ associated with the operator $\partial_t - \El$ can be decomposed as
\begin{align*}
G(\xx,\xt,y,t) = \frac{1}{4\pi t\sqrt{d_\perp \theta}}e^{-\frac{|\xx-\xt+\alpha t|^2}{4\theta t} - \frac{|y|^2}{4d_\perp t}} u_\infty'(\xx) u_{\ad} (\xt)^* + \G(\xx,\xt,y,t)
\end{align*}
where $\G(\xx,\xt,y,t)$ enjoys the estimates
\begin{align*}
\left\|\G(\xx,\xt,y,t)\right\| &\leq Ct^{-1}(1+t)^{-\frac{1}{2}}e^{-\frac{|\xx-\xt + \alpha t|^2 + |y|^2}{Mt}},\\
\left\|D_{\xx,\xt,y}^\mathfrak{a} \G(\xx,\xt,y,t)\right\| &\leq Ct^{-\frac{3}{2}}e^{-\frac{|\xx-\xt + \alpha t|^2 + |y|^2}{Mt}},
\end{align*}
where $\mathfrak{a} \in \Z_{\geq 0}^3$ is a multi-index with $|\mathfrak{a}| = 1$.
\end{lem}
\begin{proof} Throughout this proof, we denote by $C>1$ a constant, which is independent of $\xx,\xt,y$ and $t$, that will be taken larger if necessary.

Without loss of generality we may assume that $\{\nu \in \C^2 \colon |\nu| \leq 2\epsilon\} \subset U$, where $\epsilon$ is as in (D2)-(D3) and $U \subset \C^2$ is the neighborhood of $0$ from Lemma~\ref{Lem1}. Let $\phi \colon \C^2 \to [0,1]$ be a smooth cut-off function satisfying $\phi({\pmb \nu}) = 1$ if $|{\pmb \nu}| \leq \epsilon$ and $\phi({\pmb \nu}) = 0$ if $|{\pmb \nu}| \geq 2\epsilon$. Let $\delta_z^m \colon \R^m \to \R$ be the Dirac delta function centered at $z \in \R^m$. Applying the Fourier-Bloch transform -- see~\S\ref{s2.1} -- yields $2\pi \check{\delta}^2_{(\xt,0)}({\pmb \nu},\xx) = e^{-i\nu_x\xt}\delta^1_{\xt}(\xx)$. So, using~\eqref{Bloch} and Lemma~\ref{Lem1}, we decompose the Green's function of problem~\eqref{lin} as
\begin{align*}
G(\xx,\xt,y,t) &= \left[e^{\El t} \delta^2_{(\xt,0)}\right](\xx,y) = I + II,\\
I := \frac{1}{4\pi^2}\int_\Omega e^{i{\pmb \nu} \cdot {\pmb \xi}} \phi({\pmb \nu})P({\pmb \nu})e^{L_{{\pmb \nu}} t} \delta^1_\xt(\xx)d{\pmb \nu}, &\qquad II := \frac{1}{4\pi^2}\int_\Omega e^{i{\pmb \nu} \cdot {\pmb \xi}} (1-\phi({\pmb \nu})P({\pmb \nu}))e^{L_{{\pmb \nu}} t} \delta^1_\xt(\xx)d{\pmb \nu},
\end{align*}
where ${\pmb \xi} := (\xx-\xt,y)$ and $P({\pmb \nu}) := \langle q_\ad({\pmb \nu}),\cdot\rangle_2 q({\pmb \nu})$ is the spectral projection in $L_\per^2([0,1],\C^n)$ onto $\ker(L_{\pmb \nu} - \lambda_0({\pmb \nu}))$. By Lemma~\ref{Lem1}, $P({\pmb \nu})$ is well-defined for any ${\pmb \nu} \in U$ and depends analytically on ${\pmb \nu}$. It holds $P({\pmb \nu}) e^{L_{\pmb \nu} t} = e^{\lambda_0({\pmb \nu})t}P({\pmb \nu})$ for ${\pmb \nu} \in U$. Thus, we calculate using~\eqref{intid1} and Lemma~\ref{Lem1}
\begin{align*}
4\pi^2 I &= \int_\Omega e^{i{\pmb \nu} \cdot {\pmb \xi}} \phi({\pmb \nu})P({\pmb \nu})e^{L_{{\pmb \nu}} t} \delta^1_{\xt}(\xx)d{\pmb \nu}\\
  &= \int_{|{\pmb \nu}|_{\R^2} \leq 2\epsilon} e^{i{\pmb \nu} \cdot {\pmb \xi}} \phi({\pmb \nu})e^{\lambda_0({\pmb \nu}) t} q({\pmb \nu},\xx) q_\ad({\pmb \nu},\xt)^* d{\pmb \nu}\\
  &= \int_{\R^2} e^{i{\pmb \nu} \cdot {\pmb \xi}} e^{(i\alpha\nu_x - \theta \nu_x^2 - d_\perp \nu_y^2)t} u_\infty'(\xx) u_{\ad} (\xt)^* d{\pmb \nu} - \int_{|{\pmb \nu}|_{\R^2} \geq 2\epsilon} e^{i{\pmb \nu} \cdot {\pmb \xi}} e^{(i\alpha\nu_x - \theta \nu_x^2 - d_\perp \nu_y^2)t} u_\infty'(\xx) u_{\ad} (\xt)^* d{\pmb \nu}\\
  & \qquad + \int_{|{\pmb \nu}|_{\R^2} \leq 2\epsilon} e^{i{\pmb \nu} \cdot {\pmb \xi}} e^{(i\alpha\nu_x - \theta \nu_x^2 - d_\perp \nu_y^2)t} \left(e^{\h({\pmb \nu})t} \phi({\pmb \nu}) q({\pmb \nu},\xx) q_\ad({\pmb \nu},\xt)^* - q(0,\xx) q_\ad(0,\xt)^*\right) d{\pmb \nu}\\
  &= \frac{\pi}{t\sqrt{d_\perp \theta}}e^{-\frac{|\xx-\xt+\alpha t|^2}{4\theta t} - \frac{|y|^2}{4d_\perp t}} u_\infty'(\xx) u_{\ad} (\xt)^* + II' + III'.
\end{align*}
Let $\mathfrak{b} \in \Z_{\geq 0}^3$ with $0 \leq |\mathfrak{b}| \leq 1$. Using~\eqref{intid2} we estimate
\begin{align*} \left\|D^\mathfrak{b}_{\xx,\xt,y} II'\right\| &\leq C\int_{|{\pmb \nu}|_{\R^2} \geq 2\epsilon} \left(1+|\nu_x| + |\nu_y|\right)^{|\mathfrak{b}|} e^{-(\theta \nu_x^2 + d_\perp \nu_y^2)t} d{\pmb \nu}\\
&\leq C \int_{-2\epsilon}^{2\epsilon} \int_{\R} \left(1+|\nu_x| + |\nu_y|\right)^{|\mathfrak{b}|} e^{-\theta \nu_x^2t}d\nu_x e^{-d_\perp\nu_y^2t} d\nu_y \\
& \qquad + C\int_{\R} \int_{-2\epsilon}^{2\epsilon} \left(1+|\nu_x| + |\nu_y|\right)^{|\mathfrak{b}|} e^{-\theta \nu_x^2t}d\nu_x e^{-d_\perp\nu_y^2t} d\nu_y\\
&\leq Ct^{-1}\left(1 + t^{-\frac{|\mathfrak{b}|}{2}}\right)e^{-2\epsilon^2 \min\{d_\perp,\theta\} t},\\
&\leq C t^{-1-\frac{|\mathfrak{b}|}{2}}e^{-m t}e^{-\frac{|\xx-\xt + \alpha t|^2 + |y|^2}{Mt}},
\end{align*}
for some $m > 0$ and $M > 1$, using that $(|\xx-\xt+\alpha t|+|y|)/t$ is bounded in the latter inequality.

The estimation of $III'$ is more elaborate. As in~\cite{OHZUM} define
\begin{align*} a_1 := \frac{\xx - \xt + \alpha t}{2\theta t}, \quad a_2 := \frac{y}{2d_\perp t}, \quad \tilde{a}_{1,2} := \mathrm{sgn}(a_{1,2})\min\{2\epsilon,|a_{1,2}|\} ,\end{align*}
and abbreviate
\begin{align*} h(\nu_x,\nu_y,\xx,\xt,y,t) := e^{i{\pmb \nu} \cdot {\pmb \xi}} e^{(i\alpha\nu_x - \theta \nu_x^2 - d_\perp \nu_y^2)t} \left(e^{\h({\pmb \nu})t} \phi({\pmb \nu}) q({\pmb \nu},\xx) q_\ad({\pmb \nu},\xt)^* - q(0,\xx) q_\ad(0,\xt)^*\right).\end{align*}
We consider $III'$ as a double complex line integral. By Cauchy's integral theorem the integral of $h(\cdot,\nu_y,\xx,\xt,t)$ over $[-2\epsilon,2\epsilon]$ equals the integral over the other three sides of the rectangle in $\C$ formed by the points $\pm 2\epsilon$, $\pm 2\epsilon + i\tilde{a}_1$. Denote by $\Gamma_1 \subset \C$ the curve consisting of these three sides and define $\Gamma_2 \subset \C$ similarly by the points $\pm 2\epsilon$, $\pm 2\epsilon + i\tilde{a}_2$. Using Cauchy's theorem twice we rewrite $III'$ as
\begin{align}
D_{\xx,\xt,y}^\mathfrak{a} III' &= \int_{\Gamma_1} \int_{\Gamma_2} D_{\xx,\xt,y}^\mathfrak{a} h(\nu_x,\nu_y,\xx,\xt,y,t)d\nu_yd\nu_x. \label{CauchyIII}
\end{align}
Our next step is to bound the integrand $D_{\xx,\xt,y}^\mathfrak{a} h(\nu_x,\nu_y,\xx,\xt,y,t)$ as a complex function of ${\pmb \nu} \in \Gamma_1 \times \Gamma_2$. For ${\pmb \nu} \in \Gamma_1 \times \Gamma_2 \subset U$ we bound using Lemma~\ref{Lem1}
\begin{align}
\begin{split}
\left\|e^{\h({\pmb \nu})t} \right.\!&\!\left.\phi({\pmb \nu}) D_{\xx,\xt,y}^\mathfrak{a}\left(q({\pmb \nu},\xx) q_\ad({\pmb \nu},\xt)^*\right) - D_{\xx,\xt,y}^\mathfrak{a} \left(q(0,\xx) q_\ad(0,\xt)^*\right)\right\| \\
&\leq \left|e^{\h({\pmb \nu})t} - 1\right|\left\|\phi({\pmb \nu}) D_{\xx,\xt,y}^\mathfrak{a}\left(q({\pmb \nu},\xx) q_\ad({\pmb \nu},\xt)^*\right)\right\| \\
 &\qquad + \left\|\phi({\pmb \nu}) D_{\xx,\xt,y}^\mathfrak{a} \left(q({\pmb \nu},\xx) q_\ad({\pmb \nu},\xt)^*\right) - D_{\xx,\xt,y}^\mathfrak{a}\left(q(0,\xx) q_\ad(0,\xt)^*\right)\right\|\\
&\leq C\left(e^{|\h({\pmb \nu})| t} - 1 + |{\pmb \nu}|\right)\\
&\leq C\left(|\h({\pmb \nu})| t e^{|\h({\pmb \nu})|t} + |\nu_y| + |\nu_x|\right)\\
&\leq C\left(\left(|\nu_x| + |\nu_y|\right)^3 t + |\nu_y| + |\nu_x|\right) e^{c_0 \epsilon (\nu_x^2 + \nu_y^2) t},
\end{split}\label{estinteg}
\end{align}
where $c_0$ is a constant independent of $\epsilon$. By taking $\epsilon > 0$ smaller if necessary, we may assume $2c_0\epsilon \leq \theta, d_\perp$. We apply the triangle inequality to~\eqref{CauchyIII} and bound the integrand using~\eqref{estinteg}. This leads to a sum of terms of the form
\begin{align}
\begin{split}
&C t^m \int_{\Gamma_1} \int_{\Gamma_2} \left|e^{i{\pmb \nu} \cdot {\pmb \xi}} e^{(i\alpha\nu_x - \theta \nu_x^2 - d_\perp \nu_y^2)t + c_0 \epsilon (\nu_x^2 + \nu_y^2) t} |\nu_x|^k |\nu_y|^l \right|d{\pmb \nu}\\
&= C t^m \int_{\Gamma_1}  \left|e^{i\nu_x (\xx - \xt + \alpha t) - (\theta - c_0\epsilon) \nu_x^2 t}\right| |\nu_x|^k d\nu_x \int_{\Gamma_2} \left|e^{i\nu_y y - (d_\perp - c_0\epsilon) \nu_y^2 t} \right| |\nu_y|^l d\nu_y,
\end{split}
\label{summand}
\end{align}
with integers $k, l, m \geq 0$ satisfying $k + l - 2m \geq 1$. Using~\eqref{intid2} and the identities
\begin{align*}
-\tilde{a}_2y &\leq -2d_\perp \tilde{a}_2^2 t,   \quad -zy \leq -2d_\perp |z| |\tilde{a}_2|t \leq -2d_\perp z^2t, \quad z \in [0,\tilde{a}_2],
\end{align*}
we estimate the second factor in~\eqref{summand}
\begin{align*}
& \int_{\Gamma_2} \left|e^{i\nu_y y - (d_\perp - c_0\epsilon) \nu_y^2 t} \right| |\nu_y|^l d\nu_y,\\
& \quad \leq C \int_{-2\epsilon}^{2\epsilon} \left|e^{i(z + i\tilde{a}_2) y - (d_\perp - c_0\epsilon) (z+i\tilde{a}_2)^2t}\right| |z+i\tilde{\alpha_2}|^l dz + C\int_0^{\tilde{a}_2} \left|e^{i(2\epsilon + iz) y - (d_\perp - c_0\epsilon) (2\epsilon + iz)^2t}\right| |2\epsilon + iz|^l dz \\
&\quad \leq C e^{-d_\perp \tilde{a}_2^2 t/2} \sum_{j = 0}^l\int_{-2\epsilon}^{2\epsilon} e^{-(d_\perp - c_0\epsilon)z^2 t} |z|^{l-j}|\tilde{\alpha_2}|^{j} dz + Ce^{-2d_\perp \epsilon^2 t} \sum_{j = 0}^l\int_0^{|\tilde{a}_2|} e^{-(d_\perp - c_0\epsilon) z^2 t} |z|^{l-j}\epsilon^{j} dz \\
&\quad \leq C  e^{-d_\perp \tilde{a}_2^2 t/2} \sum_{j = 0}^l\min\left\{1,t^{-(l+1-j)/2} |\tilde{\alpha_2}|^{j}\right\} + e^{-2d_\perp \epsilon^2 t/2}\sum_{j = 0}^l \min\left\{1, t^{-(l+1-j)/2} \epsilon^{j}\right\}\\
&\quad \leq C  e^{-d_\perp \tilde{a}_2^2 t/4} \min\left\{1, t^{-(l+1)/2} \right\} + e^{-d_\perp \epsilon^2 t} \min\left\{1, t^{-(l+1)/2}\right\}\\
&\quad \leq C (1+t)^{-(1+l)/2} \left(e^{-d_\perp \tilde{a}_2^2 t/4} + e^{-d_\perp \epsilon^2 t}\right)\\
&\quad \leq C (1+t)^{-(1+l)/2} e^{-\frac{|y|^2}{Mt}},
\end{align*}
for some $M > 0$, using that $(|\xx-\xt+\alpha t|+|y|)/t$ is bounded in the latter inequality. Similarly, one establishes the bound on the second factor in~\eqref{summand}
\begin{align*}
\int_{\Gamma_1}  \left|e^{i\nu_x (\xx - \xt + \alpha t) - (\theta - c_0\epsilon) \nu_x^2 t}\right| |\nu_x|^k d\nu_x \leq C (1+t)^{-(1+k)/2} e^{-\frac{|\xx-\xt + \alpha t|^2}{Mt}},
\end{align*}
for some $M > 0$. We employ the latter two estimates to bound the summands of the form~\eqref{summand} in order to obtain
\begin{align*} \left\|D_{\xx,\xt,y}^\mathfrak{a} III'\right\| &\leq  \int_{\Gamma_1} \int_{\Gamma_2} \left\| D_{\xx,\xt,y}^\mathfrak{a} h(\nu_x,\nu_y,\xx,\xt,y,t)\right\|d\nu_yd\nu_x \leq C(1+t)^{-3/2} e^{-\frac{|\xx-\xt+\alpha t|^2 + |y|^2}{Mt}},\end{align*}
using $k+l-2m\geq 1$ holds in~\eqref{summand}.

Next, we split $II$ into two integrals
\begin{align*} 4\pi^2 II &= \int_{|{\pmb \nu}|_{\R^2} \leq \epsilon} e^{i{\pmb \nu} \cdot {\pmb \xi}} (1-\phi({\pmb \nu})P({\pmb \nu}))e^{L_{{\pmb \nu}} t} \check{\delta}_{\xt}({\pmb \nu},\xx)d{\pmb \nu} + \int_{{\pmb \nu} \in \Omega, |{\pmb \nu}|_{\R^2} \geq \epsilon} e^{i{\pmb \nu} \cdot {\pmb \xi}} (1-\phi({\pmb \nu})P({\pmb \nu}))e^{L_{{\pmb \nu}} t} \check{\delta}_{\xt}({\pmb \nu},\xx)d{\pmb \nu}\\
&= \widetilde{I} + \widetilde{II},
\end{align*}
We start estimating $\widetilde{II}$. Note that by assumption (D3) the spectrum $\sigma(L_{\pmb \nu})$ is confined to $\Re(\lambda) < -\eta$ for ${\pmb \nu} \in \Omega$ with $|{\pmb \nu}| \geq \epsilon$. Combining this with Corollary~\ref{high3} yields that there exists $\kappa > 0$, independent of ${\pmb \nu}$ and $\lambda$, such that the contour
\begin{align} \Gamma_{\pmb \nu} := \{\lambda \in \C \colon \Re(\lambda) = -\kappa(1 + \nu_y^2 + |\Im(\lambda)|)\}, \label{contour}\end{align}
is contained in the resolvent set of $L_{\pmb \nu}$ for ${\pmb \nu} \in \Omega$ with $|{\pmb \nu}| \geq \epsilon$. Using the Laplace transform we rewrite $\widetilde{II}$ as
\begin{align*} \widetilde{II} = \frac{1}{2\pi i} \int_{{\pmb \nu} \in \Omega, |{\pmb \nu}| \geq \epsilon} \int_{\Gamma_{\pmb \nu}} e^{i{\pmb \nu} \cdot {\pmb \xi} + \lambda t} (1-\phi({\pmb \nu})P({\pmb \nu})) \G_{{\pmb \nu},\lambda}(\xx,\xt)d\lambda d{\pmb \nu},\end{align*}
where $\G_{{\pmb \nu},\lambda}(\xx,\xt)$ denotes the Green's function associated with the operator $L_{\pmb \nu} - \lambda$. Let $j \in \{0,1\}$. Using Corollary~\ref{high3}, Remark~\ref{mid} and the integral identities~\eqref{intid2} and~\eqref{intid3} we estimate
\begin{align*} \left\|\partial_y^j \ \widetilde{II}\ \right\| &\leq C \int_{{\pmb \nu} \in \Omega, |{\pmb \nu}| \geq \epsilon} \int_{\Gamma_{\pmb \nu}} |\nu_y|^j e^{\Re(\lambda) t} \left\|\G_{{\pmb \nu},\lambda}(\xx,\xt)\right\|d\lambda d{\pmb \nu}\\
&\leq C e^{-\kappa t} \int_{\R} \int_0^\infty \frac{|\nu_y|^j}{\sqrt{|\Im(\lambda)|}} e^{-\kappa t(|\Im(\lambda)|  + \nu_y^2)}d|\Im(\lambda)| d\nu_y\\
&\leq C t^{-1- \frac{j}{2}} e^{-\kappa t}\\
&\leq C t^{-1 - \frac{j}{2}}e^{-m t}e^{-\frac{|\xx-\xt + \alpha t|^2 + |y|^2}{Mt}},
\end{align*}
for some $m > 0$ and $M > 1$, using that $(|\xx-\xt+\alpha t|+|y|)/t$ is bounded in the latter inequality. Let $\mathfrak{c} \in \Z_{\geq 0}^2$ be a multi-index with $|\mathfrak{c}| = 1$. Analogously, we estimate
\begin{align*} \left\|D_{\xx,\xt}^\mathfrak{c} \ \widetilde{II} \ \right\| &\leq C \int_{{\pmb \nu} \in \Omega, |{\pmb \nu}| \geq \epsilon} \int_{\Gamma_{\pmb \nu}} e^{\Re(\lambda) t} \left(|\nu_x|\left\|\G_{{\pmb \nu},\lambda}(\xx,\xt)\right\|+\left\|D_{\xx,\xt}^\mathfrak{c} \G_{{\pmb \nu},\lambda}(\xx,\xt)\right\|\right)d\lambda d{\pmb \nu}\\
&\leq C e^{-\kappa t} \int_{\R} \int_0^\infty \left(1 + \frac{1}{\sqrt{|\Im(\lambda)|}} \right) e^{-\kappa t(|\Im(\lambda)|  + \nu_y^2)}d|\Im(\lambda)| d\nu_y\\
&\leq C t^{-1}\left(1 + t^{-1/2}\right) e^{-\kappa t}\\
&\leq C t^{-3/2} e^{-m t}e^{-\frac{|\xx-\xt + \alpha t|^2 + |y|^2}{Mt}},
\end{align*}
for some $m > 0$ and $M > 1$.

For the estimation of $\widetilde{I}$ we note that $\phi({\pmb \nu}) = 1$ for $|{\pmb \nu}|_{\R^2} \leq \epsilon$. For $|{\pmb \nu}|_{\R^2} \leq \epsilon$ the spectrum of the operator $L_{\pmb \nu}$ restricted to $\ker(P({\pmb \nu}))$ is confined to $\Re(\lambda) < -\eta$ by Lemma~\ref{Lem1} (by taking $\eta$ smaller if necessary), since $P({\pmb \nu})$ is the spectral projection associated with the critical eigenvalue $\lambda_0({\pmb \nu})$ of $L_{\pmb \nu}$. Therefore, we may assume (by adapting $\kappa > 0$ if necessary) that the contour $\Gamma_{\pmb \nu}$ defined in~\eqref{contour} lies in the resolvent set of $L_{\pmb \nu}|_{\ker(P({\pmb \nu}))}$ for all $|{\pmb \nu}|_{\R^2} \leq \epsilon$ by Corollary~\ref{high3}. Estimating $\widetilde{I}$ is therefore similar to estimating $\widetilde{II}$. Let $\widetilde{\G}_{{\pmb \nu},\lambda}(\xx,\xt)$ be the Green's function associated with the operator $L_{\pmb \nu}|_{\ker(P({\pmb \nu}))} - \lambda$. Let $K$ be some compact set in $\Omega \times \C$ such that for pairs $({\pmb \nu},\lambda) \in K$ it holds $\lambda \in \rho(L_{\pmb \nu}|_{\ker(P({\pmb \nu}))})$. By analyticity of the Green's function in $\lambda$ and ${\pmb \nu}$ it holds
\begin{align*} \sup_{\begin{smallmatrix} ({\pmb \nu},\lambda) \in K,\\ \xx,\xt \in [0,1]\end{smallmatrix}} \left\|D_{\xx,\xt}^\mathfrak{c}\widetilde{\G}_{{\pmb \nu},\lambda}(\xx,\xt)\right\| \leq C,\end{align*}
where $\mathfrak{c} \in \Z_{\geq 0}^2$ is a multi-index with $0 \leq |\mathfrak{c}| \leq 1$. Thus, using the Laplace transform, Corollary~\ref{high3} and the integral identities~\eqref{intid2} and~\eqref{intid3} we estimate
\begin{align*} \left\|\partial_y^j \ \widetilde{I} \ \right\| &\leq C \int_{|{\pmb \nu}|_{\R^2} \leq \epsilon} \int_{\lambda \in \Gamma_{\pmb \nu}, |\lambda| \leq R} |\nu_y|^j e^{\Re(\lambda) t} \left\|\widetilde{\G}_{{\pmb \nu},\lambda}(\xx,\xt)\right\|d\lambda d{\pmb \nu}\\ &\qquad + C\int_{|{\pmb \nu}|_{\R^2} \leq \epsilon} \int_{\lambda \in \Gamma_{\pmb \nu}, |\lambda| > R} |\nu_y|^j e^{\Re(\lambda) t} \left\|\G_{{\pmb \nu},\lambda}(\xx,\xt)\right\|d\lambda d{\pmb \nu}\\
&\leq Ce^{-\kappa t} \int_{\R} \int_0^\infty |\nu_y|^j  \left(1 + \frac{1}{\sqrt{|\Im(\lambda)|}}\right)  e^{-\kappa t(|\Im(\lambda)|  + \nu_y^2)}d|\Im(\lambda)| d\nu_y\\
&\leq C t^{-1 - \frac{j}{2}}e^{-m t}e^{-\frac{|\xx-\xt + \alpha t|^2 + |y|^2}{Mt}},
\end{align*}
for some $m > 0$ and $M > 1$, using that $(|\xx-\xt+\alpha t|+|y|)/t$ is bounded in the latter inequality. Analogously, we estimate
\begin{align*} \left\|D_{\xx,\xt}^\mathfrak{c}\ \widetilde{I} \ \right\| &\leq C \int_{|{\pmb \nu}|_{\R^2} \leq \epsilon} \int_{\lambda \in \Gamma_{\pmb \nu}, |\lambda| \leq R} e^{\Re(\lambda) t} \left(|\nu_x|\left\|\widetilde{\G}_{{\pmb \nu},\lambda}(\xx,\xt)\right\| + \left\|D_{\xx,\xt}^\mathfrak{c} \widetilde{\G}_{{\pmb \nu},\lambda}(\xx,\xt)\right\|\right)d\lambda d{\pmb \nu}\\ &\qquad + C\int_{|{\pmb \nu}|_{\R^2} \leq \epsilon} \int_{\lambda \in \Gamma_{\pmb \nu}, |\lambda| > R} e^{\Re(\lambda) t}  \left(|\nu_x|\left\|\G_{{\pmb \nu},\lambda}(\xx,\xt)\right\|+\left\|D_{\xx,\xt}^\mathfrak{c} \G_{{\pmb \nu},\lambda}(\xx,\xt)\right\|\right)d\lambda d{\pmb \nu}\\
&\leq C e^{-\kappa t} \int_{\R} \int_0^\infty  \left(1 + \frac{1}{\sqrt{|\Im(\lambda)|}}\right) e^{-\kappa t(|\Im(\lambda)|  + \nu_y^2)}d|\Im(\lambda)| d\nu_y\\
&\leq C t^{-\frac{3}{2}}e^{-m t}e^{-\frac{|\xx-\xt + \alpha t|^2 + |y|^2}{Mt}},
\end{align*}
for some $m > 0$ and $M > 1$.
\end{proof}

\subsection{Proof of Theorem~\ref{theope}}

We employ Lemmas~\ref{Point1} and~\ref{Point2} to prove Theorem~\ref{theope}. First, observe that a direct calculation establishes the estimates in~\eqref{pointwiseestimates} on $e(\xx,\xt,y,t)$ and its derivatives. Now, let $S > 1$ be as in Lemma~\ref{Point1}. It holds by Lemma~\ref{Point2}
\begin{align*}
\GT(\xx,\xt,y,t) = \begin{cases} \G(\xx,\xt,y,t) + \frac{1-\chi(t)}{4\pi t\sqrt{d_\perp \theta}}e^{-\frac{|\xx-\xt+\alpha t|^2}{4\theta t} - \frac{|y|^2}{4d_\perp t}} u_\infty'(\xx) u_{\ad} (\xt)^* & \text{if } |\xx-\xt+\alpha t|+|y| \leq St,\\
G(\xx,\xt,y,t) - e(\xx,\xt,y,t) & \text{if } |\xx-\xt+\alpha t|+|y| > St,
\end{cases}
\end{align*}
for $\xx,\xt,y \in \R$ and $t > 0$. The bounds in~\eqref{pointwiseestimates} on $G(\xx,\xt,y,t)$ and its derivatives follow for $\xx,\xt,y \in \R, t > 0$ satisfying $|\xx-\xt+\alpha t|+|y| \leq St$  from Lemma~\ref{Point1} and the fact that $t \mapsto 1 - \chi(t)$ vanishes for $t \geq 2$. The bounds in~\eqref{pointwiseestimates} on $G(\xx,\xt,y,t)$ and its derivatives follow for $\xx,\xt,y \in \R, t > 0$ satisfying $|\xx-\xt+\alpha t|+|y| > St$ from Lemma~\ref{Point2} and the estimate
\begin{align*} \left\|D_{\xx,\xt,y}^\mathfrak{a} e(\xx,\xt,y,t)\right\| &\leq C(1+t)^{-1} e^{-\frac{|\xx-\xt + \alpha t|^2 + |y|^2}{Mt}} \leq C(1+t)^{-1}e^{-\frac{S^2}{2M}t} e^{-\frac{|\xx-\xt + \alpha t|^2 + |y|^2}{2Mt}},\end{align*}
where $C,M > 1$ are some constants, which are independent of $\xx,\xt,y$ and $t$. $\hfill \Box$

\bibliographystyle{plain}
\bibliography{mybib2}

\end{document}